\documentclass{amsart}

\usepackage{amsfonts}
\usepackage{amsmath}

\usepackage{datetime}
\usepackage{url}
\usepackage[top=3cm, bottom=3cm, left=2.5cm, right=2.5cm]{geometry}

\let\underbar\underline

\newtheorem{theorem}{Theorem}
\theoremstyle{plain}

\newtheorem{lemma}[theorem]{Lemma}

\newtheorem{proposition}[theorem]{Proposition}
\newtheorem{remark}[theorem]{Remark}

\newtheorem{appendixTheorem}{Theorem}[section]
\newtheorem{appendixLemma}[appendixTheorem]{Lemma}
\newtheorem{appendixRemark}[appendixTheorem]{Remark}

\theoremstyle{definition} 
\newtheorem{appendixDefinition}[appendixTheorem]{Definition}
\newtheorem{definition}[theorem]{Definition}

\newcommand{\R}{\mathbb{R}}
\newcommand{\N}{\mathbb{N}}
\newcommand{\E}{\mathbb{E}}
\newcommand{\F}{\mathcal{F}}
\renewcommand{\P}{\mathbb{P}}

\newcommand{\rp}{\zeta}

\newcommand{\tY}{\tilde{Y}}
\newcommand{\tZ}{\tilde{Z}}

\newcommand{\ty}{ {\tilde{y}} } 
\newcommand{\tz}{ {\tilde{z}} }

\newcommand{\tu}{ {\tilde{y}} } 
\newcommand{\tp}{ {\tilde{z}} }

\newcommand{\uu}{y} 
\newcommand{\pp}{z}

\newcommand{\Tr}{\operatorname{Tr}}
\newcommand{\Lip}{\operatorname{Lip}}
\newcommand{\Cunif}{C_{\operatorname{unif}}}
\newcommand{\tCunif}{\tilde{C}_{\operatorname{unif}}}
\newcommand{\pvar}{p-\operatorname{var}}

\newcommand{\itob}{b}

\input{tcilatex}
\input tcilatex

\begin{document}
\title[Rough BSDEs]{Backward stochastic differential equations with rough drivers}
\author{Joscha Diehl, Peter Friz}
\address{The first author (diehl@math.tu-berlin.de) is affiliated to TU\ Berlin, the second author (friz@math.tu-berlin.de) is affiliated to TU\ and WIAS Berlin.}

\begin{abstract}
Backward stochastic differential equations (BSDEs) in the sense of
Pardoux-Peng [Backward stochastic differential equations and
quasilinear parabolic partial differential equations, Lecture Notes in
Control and Inform. Sci., 176, 200--217, 1992] provide a
non-Markovian extension to certain classes of non-linear partial
differential equations; the non-linearity is expressed in the
so-called driver of the BSDE. Our aim is to deal with drivers which
have very little regularity in time. To this end we establish
continuity of BSDE solutions with respect to rough path metrics in the
sense of Lyons [Differential equations driven by rough signals. Rev.
Mat. Iberoamericana 14, no. 2, 215--310, 1998] and so obtain a notion
of "BSDE with rough driver". Existence, uniqueness and a version of
Lyons' limit theorem in this context are established. Our main tool,
aside from rough path analysis, is the stability theory for quadratic
BSDEs due to Kobylanski [Backward stochastic differential equations and
partial differential equations with quadratic growth. Ann. Probab.,
28(2):558--602, 2000].

\end{abstract}

\maketitle


\section{Introduction}

\bigskip We recall that \textit{backward stochastic differential equations}
(BSDEs) are stochastic equations of the type%
\begin{equation}
  Y_{t} = \xi + \int_{t}^{T} f\left( r,Y_{r},Z_{r} \right) dr
            -\int_{t}^{T}Z_{r}dW_{r}.  \label{BSDE_ModelEqu}
\end{equation}%
Here, $W$ is an $m$-dimensional Brownian motion on some filtered probability
space $\left( \Omega ,\mathcal{F},\left( \mathcal{F}_{t}\right) _{0\leq
t\leq T},\mathbb{P}\right) $. The \textit{terminal data} $\xi $ is assumed
to be $\mathcal{F}_{T}$-measurable, the \textit{driver} $f:\Omega \times \left[ 0,T\right]
\times \mathbb{R\times }\mathbb{R}^{m} \to \R$ is a predictable
random field; a solution to this equation is a $\left( 1+m\right) $%
-dimensional adapted solution process of the form $\left( Y_{t},Z_{t}\right)
_{0\leq t\leq T}$; subject to some integrability properties depending on the
framework imposed by the type of assumptions on $f$. Equation (\ref%
{BSDE_ModelEqu}) can also be written in differential form%
\begin{equation*}
  -dY_{t} = f\left( t, Y_{t}, Z_{t} \right) dt - Z_{t}dW_t.
\end{equation*}%
The aim of this paper, partially motivated from the recent progress on
partial differential equations driven by rough path
\cite{bibCaruanaFriz,bibCaruanaFrizOberhauser,bibGubinelliTindel,bibDeyaGubinelliTindel,bibTeichmann},
is to consider  
\begin{equation*}
  -dY_{t} = f\left( t,Y_{t},Z_{t} \right) dt
          + H(Y_t) d\rp_t
          - Z_{t}dW_t,
\end{equation*}%
where $\rp$ is (at first) a smooth $d$-dimensional driving signal -
accordingly $H=\left( H_{1},\dots ,H_{d}\right)$ - followed by a discussion in
which we establish \textit{rough path stability} of the solution process $\left( Y,Z \right)$
as a function of $\rp $. Note that we do \textit{not}
establish any sort of rough path stability in $W$.
Indeed when $f\equiv 0$ in \eqref{BSDE_ModelEqu}, BSDE theory reduces to martingale representation,
an intrinsically stochastic result which does not seem amenable to
a rough pathwise approach. \footnote{See however the recent work of Liang et al. \cite{bibLiangLyonsQian} in which martingale
representation is replaced by an abstract transformation.}
We are able to carry out our analysis in a framework in which
the $\omega$-dependence of the terms driven by $\rp$ factorizes through an It\^{o}
diffusion process. That is, we consider, for fixed $\left( t_0,x_0 \right) \in [0,T] \times \R^n$, 
\begin{eqnarray*}
  dX_{t} &=& \itob\left( \omega; t \right) dt + \sigma \left( \omega; t \right) dW_{t}
,\,\,\,t_{0}\leq t\leq T;\,\,\,\,X_{t_0}=x_0\in \mathbb{R}^{n}, \\
-dY_{t} &=&f\left( \omega; t,Y_{t},Z_{t} \right) dt+H\left(
X_{t},Y_{t}\right) d\rp - Z_{t}dW,\text{ \ \ } t_0\leq t\leq
T;\,\,\,\,Y_{T}=\xi \in L^{\infty }\left( \mathcal{F}_{T}\right) .
\end{eqnarray*}%
Our main-result is, under suitable conditions on $f$ and $H=\left(
H_{1},\dots ,H_{d}\right) $, that any sequence $\left( \rp ^{n}\right) $
which is Cauchy in rough path metric gives rise to a solution $\left(
Y,Z\right) $ of the \textit{BSDE with rough driver} 
\begin{equation}
  -dY_{t}=f\left( \omega; t,Y_{t},Z_{t} \right) dt+H\left( X_{t},Y_{t}\right)
d\mathbf{\rp }-Z_{t}dW_t,  \label{RoughBSDE_ModelEqu}
\end{equation}
where $\mathbf{\rp }$ denotes the (rough path)\ limit of $\left( \rp
^{n}\right) $ and where indeed $\left( Y,Z\right) $ depends only on $\mathbf{%
\rp }$ and not on the particular approximating sequence. An interesting
feature of this result, which somehow encodes the particular structure of
the above equation, is that one does \textit{not} need to construct resp.
understand the iterated integrals of $\rp $ and $W$; but only those of $%
\rp $ which i$_{{}}$s tantamout to speak of the rough path $\mathbf{\rp }
$. This is in strict contrast to the usual theory of rough differential
equations in which both $d\rp $ and $dW$ figure as driving differentials,
e.g. in equations of the form $dy = V_1(y) d\rp + V_2(y) dW$.

\bigskip 

If we specialize to a fully Markovian setting, say
$\xi = g\left( X_{T}\right) ,\, \sigma \left( \omega; t
\right) = \sigma \left ( t, X_t (\omega) \right), \itob \left( \omega; t
\right) = b \left( t, X_t (\omega) \right),$ $f\left( \omega; t,y,z
\right) =f\left( t, X_{t}\left( \omega \right) ,y,z\right) ,\,H=H\left(
X_{t},Y_{t}\right) $, we find that the solution to (\ref{RoughBSDE_ModelEqu}%
), evaluated at $t=t_{0}$, yields a solution to the (terminal value problem
of the) \textit{rough partial differential equation}%
\begin{equation*}
  -du=\left( \mathcal{L}u\right) dt + f\left( t,x,u,Du\ \sigma(t,x) \right) dt
    +H\left( x,u\right) d\mathbf{\rp },\,\,\,u_{T}\left( x\right) =g\left( x\right),
\end{equation*}%
where $\mathcal{L}$ denotes the generator of $X$. If one is interested in
the Cauchy problem, $\tilde{u}\left( t,x\right) =u\left( T-t,x\right) $
satisfies,%
\begin{equation}
  d\tilde{u}=\left( \mathcal{L}\tilde{u}\right) dt+f\left( x,\tilde{u}, D\tilde{%
u}\ \sigma(t,x) \right) dt+H\left( x,\tilde{u}\right) d\mathbf{\tilde{\rp}},\,\,\,\tilde{%
u}_{0}\left( x\right) =g\left( x\right),   \label{RPDE_ModelEquCauchy}
\end{equation}%
where $\mathbf{\tilde{\rp}=\rp }\left( T-\cdot _{{}}\right) $.

To the best of our knowledge, (\ref{RoughBSDE_ModelEqu}) is the first
attempt to introduce rough path methods
\cite{bibLyons94,bibLyonsQian02,bibLyonsCaruanaLevy04,bibFrizVictoir}
in the field of backward stochastic differential equations
\cite{bibPardouxPeng96,bibElKarouiPengQuenez,bibKobylanski}.
Of course, there are many hints in the
literature towards the possibility of doing so: we mention in particular the
Pardoux-Peng \cite{bibPardouxPengBDSDEs} theory of \textit{backward doubly stochastic differential
equations} (BDSDEs) which  amounts to replacing $d\mathbf{\rp }$ in (%
\ref{RoughBSDE_ModelEqu}) by another set of Brownian differentials, say $dB$%
, independent of $W$. This theory was then employed by Buckdahn and Ma \cite{bibBuckdahnMaI} to
construct (stochastic viscosity) solutions to (\ref{RPDE_ModelEquCauchy}) with $d%
\mathbf{\rp }$ replaced by a Brownian differential and the assumption
that the vector fields $H_{1}\left( x,\cdot \right) ,\dots ,H_{d}\left(
x,\cdot \right) $ commute.  

This paper is structured as follows.
In Section \ref{sectBSDEWithRoughDriver} we state and prove our main result concerning
the existence and uniqueness of BSDEs with rough drivers.
Section \ref{sectTheMarkovianSetting} specializes the setting to a purely Markovian one.
In this context BSDEs with rough drivers are connected to
rough partial differential equations, which we analyze in their own right.
In Section \ref{sectConnectiontoBDSDEs} we establish the connection to BDSDEs.

\bigskip 

\section{BSDE With Rough Driver}
\label{sectBSDEWithRoughDriver}

We fix once and for all a filtered probability space $(\Omega, \F, (\F)_t, \P)$, which carries a
$m$-dimensional Brownian motion $W$. Let $\F_t$ be the usual filtration of $W$.
Denote by $H^2_{[0,T]}(\R^m)$ the space of predictable processes $X$ in $\R^m$
such that $||X||^2 := \E[ \int_0^T |X|^2_r dr ] < \infty$.
Denote by $H^{\infty}_{[0,T]}(\R)$ the space of predictable processes that are almost surely bounded
with the topology of $\P$-a.s. convergence uniformly on $[0,T]$.
For a random variable $\xi$ we denote by $||\xi||_{\infty}$ its essential supremum,
for a process $Y$ we denote by $||Y||_{\infty}$ the essential supremum of $\sup_{0\le t\le T} |Y_t|$.

For a smooth path $\rp$ in $\R^d$ and $\xi \in L^{\infty}(\F_T)$ we consider the BSDE
\begin{align}
  \label{eqBSDEUntransformedSmooth}
  Y_t
  &= \xi + \int_t^T f(r, Y_r, Z_r) dr + \int_t^T H( X_r, Y_r ) d\rp(r) - \int_t^T Z_r dW_r, \quad t\le T,
\end{align}
where the $\R^n$-valued diffusion $X$ has the form
\begin{align*}
  X_t &= x + \int_0^t \sigma_r dW_r + \int_0^t \itob_r dr.
\end{align*}
Here,
$H = (H_1, \dots, H_d)$ with
$H_k: \R^n \times \R \to \R, k=1, \dots, d$ and
$\int_t^T H(X_r,Y_r) d\rp(r) := \sum_{k=1}^d \int_t^T H_k(X_r,Y_r) \dot{\rp}^k(r) dr$.
$W$ is an $m$-dimensional Brownian motion (hence $Z$ is a row vector taking values in $\R^{m\times 1}$ identified with $\R^m$).
$f: \Omega \times [0,T] \times \R \times \R^m \to \R$ is a predictable random function,
$x \in \R^n$,
$\sigma$ is a predictable process taking values in $\R^{n \times m}$,
$\itob$ is a predictable process taking values in $\R^n$.

\begin{definition}
  We call equation \eqref{eqBSDEUntransformedSmooth} \textit{BSDE with data $( \xi, f, H, \rp )$}.
\end{definition}

For a vector $x$ we denote the Euclidean norm as usual by $|x|$.
For a matrix $X$ we denote by $|X|$, depending on the situation,
either
the $1$-norm (operator norm),
the $2$-norm (Euclidean norm) or
the $\infty$-norm (operator norm of the transpose).
This slight abuse of notation will not lead to confusion,
as all inequalities will be valid up to multiplicative constants.

We introduce the following assumptions:

\begin{itemize}
  \item[(A1)]
    There exists a constant $C_\sigma > 0$ such that for $t \in [0,T]$
    \begin{align*}
      |\sigma_t(\omega)| \le C_\sigma \quad \P-a.s.
    \end{align*}

  \item[(A2)]
    There exists a constant $C_\itob > 0$ such that for $t \in [0,T]$
    \begin{align*}
      |\itob_t(\omega)| \le C_\itob \quad \P-a.s.
    \end{align*}

  \item[(F1)]
    There exists a constant $C_{1,f} > 0$ such that for $(t,y,z) \in [0,T]\times\R^n\times\R^m$
    \begin{align*}
      &|f(\omega;t,y,z)| \le C_{1,f} + C_{1,f} |z|^2 \quad \P-a.s., \\
      &|\partial_z f(\omega;t,y,z)| \le C_{1,f} + C_{1,f} |z| \quad \P-a.s.
    \end{align*}

  \item[(F2)]
    There exists a constant $C_{2,f} > 0$ such that for $(t,y,z) \in [0,T]\times\R^n\times\R^m$
    \begin{align*}
      \partial_y f(\omega;t,y,z) \le C_{2,f} \quad \P-a.s.
    \end{align*}
\end{itemize}

For given real numbers $\gamma > p \ge 1$ we have the following assumption:
\begin{itemize}

  \item[($H_{p,\gamma}$)]
    Let $H(x,\cdot) = (H_1(x,\cdot), \dots, H_d(x,\cdot))$ be a collection of vector fields on $\R$,
    parameterized by $x \in \R^n$.
    Assume that for some $C_H > 0$, we have joint regularity of the form
    \begin{align*}
      \sup_{i=1,\dots,d} |H_i|_{\Lip^{\gamma+2}(\R^{n+1})} \le C_H.
    \end{align*}
\end{itemize}

As a consequence of Theorem 2.3 and Theorem 2.6 in \cite{bibKobylanski}, we get the following
\begin{lemma}
  \label{lemBSDEExistence}
  Assume (A1), (A2), (F1), (F2) and let $H$ be Lipschitz on $\R^n \times \R$.
  Let $\xi \in L^{\infty}(\F_T)$ and a smooth path $\rp$ be given and let $\phi$ be the corresponding flow defined in \eqref{eqFlow}.
  Then there exists a unique solution to the BSDE with data $(\xi, f, H, \rp)$.
\end{lemma}

We want to give meaning to equation \eqref{eqBSDEUntransformedSmooth}, where the smooth path $\rp$ is replaced by
a general geometric rough path $\mathbf{\rp} \in C^{\pvar}( [0,T], G^{[p]}(\R^d) )$
\footnote{ In a Brownian context one can take $2 < p < 3$ and $G^{[p]}(\R^d) \cong \R^d \oplus so(d)$ is the state space for $d$-dimensional Brownian motion and it's L\'evy area. More generally, $G^{[p]}(\R^d)$ is the "correct" state space for a geometric p-rough path; the space of such paths subject to $p$-variation regularity (in rough path sense) yields a complete metric space under $p$-variation rough path metric. Technical details of geometric rough path spaces (as found e.g. in section 9 of \cite{bibFrizVictoir}) will not be necessary for the understanding of the present paper. }.
We present our main result, the proof of which we present at the end of the section.
\begin{theorem}
  \label{thmBSDEConvergence}
  Let $p \ge 1$, $\gamma > p$
  and $\rp^n, n=1,2,\dots,$ be smooth paths in $\R^d$.
  Assume $\rp^n \to \mathbf{\rp}$ in $p$-variation, for a $\mathbf{\rp} \in C^{\pvar}( [0,T], G^{[p]}(\R^d) )$.
  Let $\xi \in L^{\infty}(\F_T)$.
  Let $f$ be a random function satisfying (F1) and (F2).
  Moreover, assume (A1), (A2) and ($H_{p,\gamma}$).
  For $n\ge 1$ denote by $(Y^n, Z^n)$ the solutions to the BSDE with data $( \xi, f, H, \rp^n )$.

  Then there exists a process $(Y, Z) \in H^{\infty}_{[0,T]} \times H^2_{[0,T]}$ such that
  \begin{align*}
    &Y^n \to Y \quad \text{ uniformly on } [0,T]\ \P-a.s., \\
    &Z^n \to Z \quad \text{ in } H^2_{[0,T]}.
  \end{align*}

  The process is unique in the sense, that it only depends on the limiting rough path $\mathbf{\rp}$
  and not on the approximating sequence.
 \newpage We write (formally \footnote{The "integral"  $\int H( X, Y ) d \mathbf{\rp}$ is {\it not} a rough integral defined in the usual rough path theory (e.g. \cite{bibLyonsQian02} or \cite{bibFrizVictoir}); regularity issues aside one misses the iterated integrals of $X$ (and thus $W$) against those of $\rp$. For what it's worth, in the present context (\ref{eqBSDEUntransformedRough}) can be taken as an implicit definition of $\int H( X, Y ) d \mathbf{\rp}$. (Somewhat similar in spirit: F\"ollmer's It\^o's integral which appears in his It\^o formula {\it sans probabilit\'e}.) More pragmatically, notation (\ref{eqBSDEUntransformedRough}) is justified {\it a posteriori} through our uniquess result; in addition it is consistent with standard BSDE notation when 
 $\mathbf{\rp}$ happens to be a smooth path. })
  
  \begin{align}
    \label{eqBSDEUntransformedRough}
    Y_t &= \xi + \int_t^T f(r, Y_r, Z_r) dr + \int_t^T H( X_r, Y_r ) d \mathbf{\rp}(r) - \int_t^T Z_r dW_r.
  \end{align}

  Moreover, the solution mapping
  \begin{align*}
    C^{\pvar}([0,T], G^{[p]}(\R^d)) \times L^{\infty}(\F_T) &\to H^{\infty}_{[0,T]} \times H^2_{[0,T]}, \\
    (\mathbf{\rp}, \xi) &\mapsto (Y, Z)
  \end{align*}
  is continuous.
\end{theorem}

The problem in showing convergence of the processes $(Y^n, Z^n)$ in the statement of the theorem
lies in the fact, that
in general the Lipschitz constants for the correspondig BSDEs will tend to infinity as $n \to \infty$.
It does not seem possible then, to directly control the solutions via
a priori bounds, a standard tool in the theory of BSDEs (see e.g. \cite{bibElKarouiPengQuenez}).
We will take another approach
and transform the BSDEs corresponding to the smooth paths $\rp^n$ into BSDEs which are easier to analayze.

We start by defining the flow
\begin{align}
  \label{eqFlow}
  \phi(t,x,y) = y + \int_t^T \sum_{k=1}^d H_k(x, \phi(r,x,y)) d \rp^k(r).
\end{align}
Let $\phi^{-1}$ be the $y$-inverse of $\phi$, then
\begin{align*}
  \phi^{-1}(t,x,y) = y - \int_t^T \sum_{k=1}^d \partial_y \phi^{-1}(r,x,y) H_k(x,y) d \rp^k(r).
\end{align*}

We have the following
\begin{lemma}
  \label{lemBSDETransformation}
  Assume (A1), (A2), (F1), (F2) and let $H$ be Lipschitz on $\R^n \times \R$.
  Let $\xi \in L^{\infty}(\F_T)$ and a smooth path $\rp$ be given and let $\phi$ be the corresponding flow defined in \eqref{eqFlow}.
  Let $(Y,Z)$ be the unique solution to the BSDE with data $(\xi, f, H, \rp)$.

  The, the process $(\tY, \tZ)$ defined as
  \begin{align*}
    \tY_t &:= \phi^{-1}(t, X_t, Y_t), \quad
    \tZ_t :=
    -
    \frac{ \partial_x \phi( t, X_t, \tY_t ) }{ \partial_y \phi(t,X_t,\tY_t) } \sigma_t
    +
    \frac{1}{ \partial_y \phi(t, X_t, \tY_t ) } Z_t,
  \end{align*}
  satisfies the BSDE
  \begin{align}
    \label{eqBSDETransformedSmooth}
    \tY_t
    =
    \xi
    +
    \int_t^T \tilde{f}(r,X_r,\tY_r,\tZ_r) dr
    -
    \int_t^T \tZ_r dW_r,
  \end{align}
  where (throughout, $\phi$ and all its derivatives will always be evaluated at $(t,x,\ty)$)
  \begin{align*}
    \tilde{f}(t,x,\ty,\tz) &:=
      \frac{1}{\partial_y \phi}
      \Bigl\{  
        f\left(t, \phi, \partial_y \phi \tz + \partial_x \phi \sigma_t \right)
        +
        \langle \partial_x \phi, \itob_t \rangle
        +
        \frac{1}{2} \Tr\left[ \partial_{xx} \phi \sigma_t \sigma_t^T \right] \\
        &\qquad \qquad
        +
        \langle \tz, \left( \partial_{xy} \phi \sigma_t \right)^T  \rangle
        +
        \frac{1}{2} \partial_{yy} \phi |\tz|^2
        \Bigr\}.
  \end{align*}
\end{lemma}
\begin{remark}
  This ("Doss-Sussman") transformation is well known and has been recently applied to
  BDSDEs \cite{bibBuckdahnMaI} and rough partial differential equations \cite{bibFrizOberhauser}. We include details for the reader's convenience. 
\end{remark}
\begin{proof}
  Denoting $\psi := \phi^{-1}$ and $\theta_r := (r, X_r, Y_r)$, we have by It\^{o} formula
  \begin{align*}
    \psi(t, X_t, Y_t)
    &=
    \xi
    -
    \int_t^T \sum_{k=1}^d \partial_y \psi( \theta_r ) H_k(X_r, Y_r) \dot{\rp^k}(r) dr
    -
    \int_t^T \langle \partial_x \psi( \theta_r ), \itob_r \rangle dr
    -
    \int_t^T \langle \partial_x \psi( \theta_r ), \sigma_r dW_r \rangle \\
    &\quad
    +
    \int_t^T \partial_y \psi( \theta_r ) f(r,Y_r,Z_r) dr
    +
    \int_t^T \sum_{k=1}^d \partial_y \psi( \theta_r ) H_k(X_r, Y_r) \dot{\rp^k}(r) dr
    -
    \int_t^T \partial_y \psi( \theta_r ) Z_r dW_r \\
    &\quad
    -
    \frac{1}{2} \int_t^T \Tr\left[ \partial_{xx} \psi( \theta_r ) \sigma_r \sigma_r^T \rangle \right] dr
    -
    \frac{1}{2} \int_t^T \partial_{yy} \psi( \theta_r ) |Z_r|^2 dr
    -
    \int_t^T \langle \partial_{xy} \psi( \theta_r ), \sigma_r Z_r^T \rangle dr \\
    &=
    \xi
    +
    \int_t^T
      \Bigl[ 
        \partial_y \psi( \theta_r ) f(r,Y_r,Z_r)
        -
        \langle \partial_x \psi( \theta_r ), \itob_r \rangle
        -
        \frac{1}{2} \Tr\left[ \partial_{xx} \psi( \theta_r ) \sigma_r \sigma_r^T \right] \\
    &\qquad \qquad \quad
        -
        \frac{1}{2} \partial_{yy} \psi( \theta_r ) |Z_r|^2
        - \langle \partial_{xy} \psi( \theta_r ), \sigma_r Z_r^T \rangle
      \Bigr] dr \\
    &\quad
    -
    \int_t^T \langle \partial_x \psi( \theta_r ) \sigma_r + \partial_y \psi( \theta_r ) Z_r, dW_r \rangle.
  \end{align*}

  Now, by deriving the identity $\psi(t,x,\phi(t,x,\ty)) = \ty$ we get
  \begin{align*}
    0 &= \partial_x \psi + \partial_y \psi \partial_x \phi, \\
    0 &= \partial_{xx} \psi
         + \partial_{yx} \psi \otimes \partial_x \phi
         + [\partial_{xy} \psi + \partial_{yy} \psi \partial_x \phi] \otimes \partial_x \phi
         + \partial_y \psi \partial_{xx} \phi \\
      &= \partial_{xx} \psi
          + 2 \partial_{xy} \psi \otimes \partial_x \phi
          + \partial_{yy} \psi \partial_x \phi \otimes \partial_x \phi
          + \partial_y \psi \partial_{xx} \phi, \\
    1 &= \partial_y \psi \partial_y \phi, \\
    0 &= \partial_{xy} \psi \partial_y \phi + \partial_{yy} \psi \partial_x \phi \partial_y \phi + \partial_y \psi \partial_{xy} \phi, \\
    0 &= \partial_{yy} \psi (\partial_y \phi)^2 + \partial_y \psi \partial_{yy} \phi.
  \end{align*}
  And hence
  \begin{align*}
    \partial_{yy} \psi
    &=
    - \frac{ \partial_{yy} \phi }{ (\partial_y \phi)^3 }, \quad
    \partial_x \psi
    = - \frac{\partial_x \phi}{\partial_y \phi}, \quad
    \partial_{xy} \psi
    =
      \frac{\partial_{yy} \phi}{(\partial_y \phi)^3} \partial_x \phi
      -
      \frac{\partial_{xy} \phi}{ (\partial_y \phi)^2 }, \\
    \partial_{xx} \psi
    &=
    2 \left[ \frac{ \partial_{yy} \phi }{ (\partial_y \phi)^3 } \partial_x \phi - \frac{ \partial_{xy} \phi }{ (\partial_y \phi)^2 } \right] \otimes \partial_x \phi
    +
    \frac{\partial_{xx} \phi}{ (\partial_y \phi)^3 } \partial_x \phi \otimes \partial_x \phi
    -
    \frac{1}{\partial_y \phi} \partial_{xx} \phi.
  \end{align*}

  If we define
  \begin{align*}
    \tY_t &:= \psi( t, X_t, Y_t ) = \phi^{-1}( t, X_t, Y_t ), \\
    \tZ_t
    &:=
    \partial_x \psi( t, X_t, Y_t ) \sigma_t + \partial_y \psi( t, X_t, Y_t ) Z_t \\
    &=
    - \frac{ \partial_x \phi( t, X_t, \tY_t ) }{ \partial_y \phi(t,X_t,\tY_t) } \sigma_t
    +
    \frac{ 1 }{ \partial_y \phi( t, X_t, \tY_t ) } Z_t,
  \end{align*}
  and ($\psi$ and its derivatives are always evaluated at $(t, x, \phi(t,x,\ty))$, $\phi$ and its derivatives are evaluated at $(t,x,\ty)$)
  \begin{align*}
    \tilde{f}(t,x,\ty,\tz)
    &:= 
    \partial_y \psi f\left(t, \phi, \partial_y \phi ( \tz + \frac{\partial_x \phi \sigma_t}{\partial_y \phi} ) \right)
    -
    \langle \partial_x \psi, \itob_t \rangle
    -
    \frac{1}{2} \Tr\left[ \partial_{xx} \psi \sigma_t \sigma_t^T \right] \\
    &\quad
    -
    \frac{1}{2} \partial_{yy} \psi | \frac{ \tz - \partial_x \psi \sigma_t }{ \partial_y \psi } |^2
    -
    \langle \partial_{xy} \psi, \sigma_t \left( \frac{ \tz - \partial_x \psi \sigma_t }{ \partial_y \psi } \right)^T  \rangle \\
    &=
    \frac{1}{\partial_y \phi}
    \Bigl\{  
      f\left(t, \phi, \partial_y \phi \tz + \partial_x \phi \sigma_t \right)
      +
      \langle \partial_x \phi, \itob_t \rangle
      +
      \frac{1}{2} \Tr\left[ \partial_{xx} \phi \sigma_t \sigma_t^T \right] \\
      &\qquad \qquad
      +
      \langle \tz, \left( \partial_{xy} \phi \sigma_t \right)^T  \rangle
      +
      \frac{1}{2} \partial_{yy} \phi |\tz|^2
      \Bigr\},
  \end{align*}
  we therefore obtain
  \begin{align*}
    \tY_t = \xi + \int_t^T \tilde{f}(r, x, \tY_r, \tZ_r) dr - \int_t^T \tZ_r dW_r.
  \end{align*}
\end{proof}

\begin{definition}
  We call equation \eqref{eqBSDETransformedSmooth}
  \textit{BSDE with data $( \xi, \tilde{f}, 0, 0 )$}.
\end{definition}

The BSDE \eqref{eqBSDEUntransformedSmooth} only makes sense for a smooth path $\rp$.
On the other hand, equation \eqref{eqFlow} yields a flow of diffeomorphisms
for a general geometric rough path $\mathbf{\rp} \in C^{\pvar}( [0,T], G^{[p]}(\R^d) ), p\ge 1$.
Hence we can, also in this case, consider the function $\tilde{f}$ from the previous lemma.
We now record important properties for this induced function.

\begin{lemma}
  \label{lemBSDEPropertiesOfFTilde}
  Let $p\ge 1$, $\mathbf{\rp} \in C^{\pvar}( [0,T], G^{[p]}(\R^d) )$ and $\gamma > p$.
  Assume (A1), (A2), (F1), (F2) and ($H_{p,\gamma}$).
  Let $\phi$ be the flow corresponding to equation \eqref{eqFlow} (now solved as a rough differential equation).
  Then the function
  \begin{align}
    \label{eqInducedFTilde}
    \tilde{f}(t,x,\ty,\tz) &:=
      \frac{1}{\partial_y \phi}
      \Bigl\{  
        f\left(t, \phi, \partial_y \phi \tz + \partial_x \phi \sigma_t \right)
        +
        \langle \partial_x \phi, \itob_t \rangle
        +
        \frac{1}{2} \Tr\left[ \partial_{xx} \phi \sigma_t \sigma_t^T \right] \\
        &\qquad \qquad
        +
        \langle \tz, \left( \partial_{xy} \phi \sigma_t \right)^T  \rangle
        +
        \frac{1}{2} \partial_{yy} \phi |\tz|^2
        \Bigr\} \notag
  \end{align}
  satisfies the following properties:
  \begin{itemize}
    \item There exists a constant $\tilde{C}_{1,f} > 0$ depending only on
      $C_\sigma$, $C_\itob$,
      $C_{1,f}$,
      $C_H$ and
      $||\mathbf{\rp}||_{\pvar;[0,T]}$
      such that
      \begin{align*}
        |\tilde{f}(t,x,\ty,\tz)| &\le \tilde{C}_{1,f} + \tilde{C}_{1,f} |\tz|^2, \\
        |\partial_\tz \tilde{f}(t,x,\ty,\tz)| &\le \tilde{C}_{1,f} + \tilde{C}_{1,f} |\tz|. 
      \end{align*}
    \item
      There exists a constant $\tCunif > 0$ that only depends on
      $C_\sigma$, $C_\itob$,
      $C_{2,f}$,
      $C_H$ and
      $||\mathbf{\rp}||_{\pvar;[0,T]}$
      such that
      for every $\varepsilon$ there exists an $h_\varepsilon>0$
      that only depends on
      $C_\sigma$, $C_\itob$,
      $C_H$ and
      $||\mathbf{\rp}||_{\pvar;[0,T]}$
      such that on $[T-h_\varepsilon, T]$ we have
      \begin{align*}
        \partial_\ty \tilde{f}(t,x,\ty,\tz) \le \tCunif + \varepsilon |\tz|^2.
      \end{align*}
  \end{itemize}
\end{lemma}
\begin{proof}
  (i). Note that
  \begin{align*}
    |\tilde{f}(t,x,\ty,\tz)|
    &\le
    |\frac{1}{\partial_y \phi}|
    \Bigl(  
      |f\left(t, \phi, \partial_y \phi \tz + \partial_x \phi \sigma_t \right)|
      +
      |\langle \partial_x \phi, \itob_t \rangle|
      +
      |\frac{1}{2} \Tr\left[ \partial_{xx} \phi \sigma_t \sigma_t^T \right]| \\
      &\qquad \qquad
      +
      |\langle \tz, \left( \partial_{xy} \phi \sigma_t \right)^T  \rangle|
      +
      |\frac{1}{2} \partial_{yy} \phi| |\tz|^2
    \Bigr) \\
    &\le
    |\frac{1}{\partial_y \phi}|
    \Bigl(  
    C_{1,f} + C_{1,f} |\partial_y \phi \tz + \partial_x \phi \sigma_t|^2
      +
      |\partial_x \phi| |\itob_t|
      +
      \frac{1}{2}
      |\partial_{xx} \phi| |\sigma_t \sigma_t^T| \\
      &\qquad \qquad
      +
      |\tz| |\partial_{xy} \phi \sigma_t|
      +
      \frac{1}{2} |\partial_{yy} \phi| |\tz|^2
    \Bigr) \\
    &\le
    |\frac{1}{\partial_y \phi}|
    \Bigl(  
    C_{1,f} + C_{1,f} 2 ( |\partial_y \phi|^2 |\tz|  + |\partial_x \phi| |\sigma_t^T| )
      +
      |\partial_x \phi| |\itob_t|
      +
      \frac{1}{2}
      |\partial_{xx} \phi| |\sigma_t|^2 \\
      &\qquad \qquad
      +
      |\tz| |\partial_{xy} \phi| |\sigma_t^T|
      +
      \frac{1}{2} |\partial_{yy} \phi| |\tz|^2
    \Bigr) \\
    &\le
    \tilde{C}_{1,f} + \tilde{C}_{1,f} |\tz|^2.
  \end{align*}
  Here we have used (A1), (A2) and (F1).
  For the boundedness of the flow and its derivatives we have used Lemma \ref{lemFlowEstimates}.
  Note that $\tilde{C}_{1,f}$ hence only depends on
  $C_\sigma$,
  $C_\itob$,
  $C_{1,f}$,
  $C_H$ and
  $||\rp||_{\pvar;[0,T]}$.

(ii). Note that
\begin{align*}
  |\partial_\tz \tilde{f}(t,x,\ty,\tz)|
  &=
  |
    \partial_z f\left(t, \phi, \partial_y \phi \tz + \partial_x \phi \sigma_t \right)
    +
    \frac{1}{\partial_y \phi}
    \Bigl(  
      \partial_{xy} \phi \sigma_t
      +
      \partial_{yy} \phi \tz
   \Bigr)
  | \\
  &\le
  C_{1,f} + C_{1,f} ( |\partial_y \phi| |\tz| + |\partial_x \phi| |\sigma_t| )
  +
  |\frac{ \partial_{xy} \phi }{ \partial_y \phi }| |\sigma_t|
  +
  |\frac{ \partial_{yy} \phi }{ \partial_y \phi }| |\tz| \\
  &\le
  \tilde{C}_{1,f} + \tilde{C}_{1,f} |\tz|. 
\end{align*}
Here we have used (A1), (A2) and (F1).
For the boundedness of the flow and its derivatives we have used Lemma \ref{lemFlowEstimates}.
Note that again, $\tilde{C}_{1,f}$ hence only depends on
$C_\sigma$,
$C_\itob$,
$C_{1,f}$,
$C_H$ and
$||\rp||_{\pvar;[0,T]}$.
Without loss of generality we can choose it to be the same constant as in the estimate for (i).

(iii). Note that
\begin{align*}
  \partial_\ty \tilde{f}(t,x,\ty,\tz)
  &=
  - \frac{\partial_{yy} \phi}{ (\partial_y \phi)^2 }
  \Bigl\{  
    f\left(t, \phi, \partial_y \phi \tz + \partial_x \phi \sigma_t \right)
    +
    \langle \partial_x \phi, \itob_t \rangle
    +
    \frac{1}{2} \Tr\left[ \partial_{xx} \phi \sigma_t \sigma_t^T \right] \\
    &\qquad \qquad
    +
    \langle \tz, \left( \partial_{xy} \phi \sigma_t \right)^T  \rangle
    +
    \frac{1}{2} \partial_{yy} \phi |\tz|^2
    \Bigr\} \\
  &\quad
  +
  \frac{1}{ \partial_y \phi }
  \Bigl\{  
    \partial_y \phi \partial_y f\left(t, \phi, \partial_y \phi \tz + \partial_x \phi \sigma_t \right)
    +
    \langle \partial_{yx} \phi, \itob_t \rangle
    +
    \frac{1}{2} \Tr\left[ \partial_{yxx} \phi \sigma_t \sigma_t^T \right] \\
    &\qquad \qquad
    +
    \langle \tz, \left( \partial_{yxy} \phi \sigma_t \right)^T  \rangle
    +
    \frac{1}{2} \partial_{yyy} \phi |\tz|^2
  \Bigr\}.
  \end{align*}
  Hence using our assumptions on $f$ we get
  \begin{align*}
  \partial_\ty \tilde{f}(t,x,\ty,\tz)
  &\le
  |\frac{\partial_{yy} \phi}{ (\partial_y \phi)^2 }|
  \Bigl\{  
  C_{2,f} + C_{2,f} |\partial_y \phi \tz + \partial_x \phi \sigma_t|^2
    +
    |\partial_x \phi| |\itob_t|
    +
    \frac{1}{2} |\partial_{xx} \phi| |\sigma_t|^2 \\
    &\qquad \qquad
    +
    |\tz| |\partial_{xy} \phi| |\sigma_t|
    +
    \frac{1}{2} |\partial_{yy} \phi| |\tz|^2
  \Bigr\} \\
  &\quad
  +
  \partial_y f\left(t, \phi, \partial_y \phi \tz + \partial_x \phi \sigma_t \right)
  +
  \frac{1}{ \partial_y \phi }
  \Bigl\{  
    |\partial_{yx} \phi| |\itob_t|
    +
    \frac{1}{2} |\partial_{yxx} \phi| |\sigma_t| \\
    &\qquad \qquad
    +
    (1 + |\tz|^2 ) |\partial_{yxy} \phi| |\sigma_t|_{op}
    +
    \frac{1}{2} \partial_{yyy} \phi |\tz|^2
  \Bigr\} \\
  &\le
  |\frac{\partial_{yy} \phi}{ (\partial_y \phi)^2 }|
  \Bigl\{  
  C_{2,f} + C_{2,f} 2 |\partial_x \phi|^2 |\sigma_t|^2
    +
    |\partial_x \phi| |\itob_t|
    +
    \frac{1}{2} |\partial_{xx} \phi| |\sigma_t|^2
    +
    |\partial_{xy} \phi| |\sigma_t|
  \Bigr\} \\
  &\quad
  +
  \partial_y f\left(t, \phi, \partial_y \phi \tz + \partial_x \phi \sigma_t \right) \\
  &\quad
  +
  \frac{1}{ \partial_y \phi }
  \Bigl\{  
    |\partial_{yx} \phi| |\itob_t|
    +
    \frac{1}{2} |\partial_{yxx} \phi| |\sigma_t|
    +
    |\partial_{yxy} \phi| |\sigma_t|
  \Bigr\} \\
  &\quad
  +
  \Bigl\{
    |\frac{\partial_{yy} \phi}{ (\partial_y \phi)^2 }|
    C_{2,f} 2 |\partial_y \phi|^2
    +
    |\frac{\partial_{yy} \phi}{ (\partial_y \phi)^2 }|
    |\partial_{xy} \phi| |\sigma_t|
    +
    |\frac{\partial_{yy} \phi}{ (\partial_y \phi)^2 }|
    \frac{1}{2} |\partial_{yy} \phi| \\
    &\qquad \qquad
    +
    \frac{1}{\partial_y \phi} |\partial_{yxy} \phi| |\sigma_t^T|
    +
    \frac{1}{\partial_y \phi} \frac{1}{2} \partial_{yyy} \phi 
  \Bigr\}
  |\tz|^2 \\
  &\le
  \tCunif
  +
  \Bigl\{
    |\frac{\partial_{yy} \phi}{ (\partial_y \phi)^2 }|
    C_{2,f} 2 |\partial_y \phi|^2
    +
    |\frac{\partial_{yy} \phi}{ (\partial_y \phi)^2 }|
    |\partial_{xy} \phi| |\sigma_t|
    +
    |\frac{\partial_{yy} \phi}{ (\partial_y \phi)^2 }|
    \frac{1}{2} |\partial_{yy} \phi| \\
    &\qquad \qquad
    +
    \frac{1}{\partial_y \phi} |\partial_{yxy} \phi| |\sigma_t^T|
    +
    \frac{1}{\partial_y \phi} \frac{1}{2} \partial_{yyy} \phi 
  \Bigr\}
  |\tz|^2,
\end{align*}
where $\tCunif$ only depends on
$C_\sigma$,
$C_\itob$,
$C_H$ and
$||\mathbf{\rp}||_{\pvar;[0,T]}$
(here we have used Lemma \ref{lemFlowEstimates} to bound the flow and its derivatives).

By (A1), (A2) $\sigma$ and $\itob$ are bounded.
Then, by the properties of the flow, the term in front of $|\tz|^2$ goes uniformly to zero as $t$ approaches $T$.
To be specific: using ($H_{p,\gamma}$) we obtain, again by Lemma \ref{lemFlowEstimates},
that for every $\varepsilon > 0$ there exists an $h_\varepsilon>0$, depending on
$C_\sigma$,
$C_\itob$,
$C_H$ and
$||\mathbf{\rp}||_{\pvar;[0,T]}$
such that on $[T-h_\varepsilon,T]$ we have
\begin{align*}
  \partial_\ty \tilde{f}(t,x,\ty,\tz)
  &\le
  \tCunif + \varepsilon |\tz|^2.
\end{align*}
\end{proof}

We are now ready to prove Theorem \ref{thmBSDEConvergence}.
\begin{proof}[Proof of Theorem \ref{thmBSDEConvergence}] 
For the sake of unified notation, the (rough BSDE) solution process $(Y,Z)$ will be written as $(Y^0,Z^0)$ in what follows; similarly, the involved rough path $\mathbf{\rp}$ will be written as  $\mathbf{\rp}^0$.

  1. \textbf{Existence}
  For $n=0,1,\dots$ denote by $\phi^n$ the flow of the ODE
  \begin{align*}
    \phi^n(t, x, y) = y + \int_t^T H( x, \phi^n(r,x,y) ) d\rp^n(r).
  \end{align*}
  (For $n=0$ we mean the rough differential equation driven by $\mathbf{\rp}^0$ ).

  By Lemma \ref{lemFlowEstimates}, we have
  for all $n\ge0$, $x \in \R^n$, that $\phi^n(t,x,\cdot)$ is a flow of $C^3$-diffeomorphisms.
  Let $\psi^n(t,x,\cdot)$ be its $y$-inverse.
  We have that $\phi^n(t,\cdot,\cdot)$ and its derivatives up to order three are bounded (Lemma \ref{lemFlowEstimates}).
  The same holds true for $\psi^n(t,\cdot,\cdot)$ and its derivatives up to order three.

  Moreover, by Lemma \ref{lemConvergenceOfFlows} we have that
  locally uniformly on $[0,T] \times \R^n \times \R$
  \begin{align}
    \label{eqConvergenceOfFlows}
    (\phi^n, \frac{1}{\partial_y \phi^n}, \partial_y \phi^n, \partial_{yy} \phi^n, \partial_x \phi^n, \partial_{xx} \phi^n, \partial_{yx} \phi^n)
    \to 
    (\phi^0, \frac{1}{\partial_y \phi^0}, \partial_y \phi^0, \partial_{yy} \phi^0, \partial_x \phi^0, \partial_{xx} \phi^0, \partial_{yx} \phi^0).
  \end{align}

  Denote for $n\ge 0$ the function
  \begin{align*}
    \tilde{f}^n(r,x,\ty,\tz) &:=
      \frac{1}{\partial_y \phi^n}
      \Bigl\{  
        f\left(t, \phi^n, \partial_y \phi^n \tz + \partial_x \phi^n \sigma_t \right)
        +
        \langle \partial_x \phi^n, \itob_t \rangle
        +
        \frac{1}{2} \Tr\left[ \partial_{xx} \phi^n \sigma_t \sigma_t^T \right] \\
        &\qquad \qquad
        +
        \langle \tz, \left( \partial_{xy} \phi^n \sigma_t \right)^T  \rangle
        +
        \frac{1}{2} \partial_{yy} \phi^n |\tz|^2
      \Bigr\}.
  \end{align*}

  Now, we have seen above that for $n\ge1$, the process
  \begin{align*}
    (\tY^n,\tZ^n)
    &:=
    L^n(Y^n, Z^n) \\
    &:=
    (
      (\phi^{n})^{-1}(\cdot, X_\cdot, Y^n_\cdot),
      - \frac{ \partial_x \phi^{n}( \cdot, X_{\cdot}, (\phi^{n})^{-1}(\cdot, X_\cdot, Y^n_\cdot) ) }{ \partial_y \phi^{n}(\cdot,X_\cdot,(\phi^{n})^{-1}(\cdot, X_\cdot, Y^n_\cdot)) } \sigma_\cdot
      +
      \frac{1}{ \partial_y \phi^{n}(\cdot, X_\cdot, (\phi^{n})^{-1}(\cdot, X_\cdot, Y^n_\cdot) ) } Z^n_\cdot
    ).
  \end{align*}
  solves the BSDE with data $(\xi, \tilde{f}^n, 0, 0)$.

  Note that although $(\xi, \tilde{f}^n, 0, 0)$ is a quadratic BSDE,
  existence and uniqueness of a solution are guaranteed for $n\ge1$
  by the fact that the mapping $L^n$ is one to one and by the existence
  of a unique solution to the untransformed BSDE (Theorem 2.3 and Theorem 2.6 in \cite{bibKobylanski}).

  For $n=0$, 
  by the properties of $\tilde{f}^0$
  demonstrated in Lemma \ref{lemBSDEPropertiesOfFTilde},
  there exists a solution $(\tY^0, \tZ^0) \in H^{\infty}_{[0,T]} \times H^2_{[0,T]}$
  to the BSDE with data $(\xi, \tilde{f}^0, 0, 0)$ by Theorem 2.3 in \cite{bibKobylanski}.
  Note that it is a priori not unique, but we will show that it is at least unique on a small time interval up to $T$.

  We now construct the process $(Y^0, Z^0)$ of the statement on subintervals of $[0,T]$.
  First of all notice that we can choose the constant $\tilde{C}_{1,f}$ appearing in Lemma \ref{lemBSDEPropertiesOfFTilde}
  uniformly for all $n\ge 0$.
  Let $M := ||\xi||_{\infty} + T \tilde{C}_{1,f}$.
  By Corollary 2.2 in \cite{bibKobylanski} we have
  \begin{align}
    \label{eqBoundOnYn}
    ||\tY^n||_{\infty} \le M, \quad n\ge 0.
  \end{align}

  Now by Lemma \ref{lemBSDEPropertiesOfFTilde}
  \begin{itemize}
    \item there exists $\tilde{C}_{1,f} > 0$ that only depends on
      $C_\sigma$, $C_\itob$,
      $C_{1,f}$,
      $C_H$ and
      $||\mathbf{\rp}||_{\pvar;[0,T]}$ such that
          \begin{align*}
            |\tilde{f}^0(t,x,y,z)| &\le \tilde{C}_{1,f} + \tilde{C}_{1,f} |z|^2, \\
            |\partial_z \tilde{f}^0(t,x,y,z)| &\le \tilde{C}_{1,f} + \tilde{C}_{1,f} |z|. 
          \end{align*}
    \item
      There exists a constant $\tCunif > 0$ that only depends on
      $C_\sigma$, $C_\itob$,
      $C_{2,f}$,
      $C_H$ and
      $||\mathbf{\rp}||_{\pvar;[0,T]}$
      such that
      for every $\varepsilon$ there exists an $h_\varepsilon>0$
      that only depends on
      $C_\sigma$, $C_\itob$,
      $C_H$ and
      $||\mathbf{\rp}||_{\pvar;[0,T]}$
      such that on $[T-h_\varepsilon, T]$ we have
      \begin{align*}
        \partial_y \tilde{f}^0(t,x,y,z) \le \tCunif + \varepsilon |z|^2.
      \end{align*}
  \end{itemize}

  Hence we can choose $h = h_{\delta(\tilde{C}_{1,f}, M)}$, such that for $t\in[T-h,T]$ we have
  \begin{align*}
    \partial_y \tilde{f}(t,x,y,z) \le \tCunif + \delta(\tilde{C}_{1,f},M) |z|^2.
  \end{align*}
  Here $\delta$ is the universal function given in the statement of Theorem \ref{thmComparison}.
  We can then apply Theorem \ref{thmComparison} to get uniqueness of our solution $(\tY^0, \tZ^0)$ on $[T-h,T]$.
  Now, as a consequence of \eqref{eqConvergenceOfFlows} we have
  \begin{align*}
    \tilde{f}^n \to \tilde{f}^0 \qquad \text{ uniformly on compacta}.
  \end{align*}

  Hence, by the argument of Theorem 2.8 in \cite{bibKobylanski} we have that on $[T-h, T]$
  \begin{align}
    \label{eqTZNConvergence}
    &\tY^n \to \tY^0 \quad \text{ uniformly on } [T-h,T]\ \P-a.s., \notag \\
    &\tZ^n \to \tZ^0 \quad \text{ in } H^2_{[T-h,T]}.
  \end{align}

  Moreover, if we define
  \begin{align}
    Y^0_t &:= \phi^0(t, X_t, \tY^0_t), \quad t \in [T-h,T], \notag \\
    Z^0_t &:=
    \partial_y \phi^0(t, X_t, \tY^0_t )
    \left[ \tZ^0_t + \frac{ \partial_x \phi^0(t, X_t, \tY^0_t) }{ \partial_y \phi^0(t, X_t, \tY^0_t) } \sigma_t \right], \quad t \in [T-h,T], \notag
 \intertext{and remembering that by construction} \notag
    Y^n_t &= \phi^n(t, X_t, \tY^n_t), \notag \\
    Z^n_t &=
    \partial_y \phi^n(t, X_t, \tY^n_t ) \left[ \tZ^n_t + \frac{ \partial_x \phi^n(t, X_t, \tY^n_t) }{ \partial_y \phi^n(t, X_t, \tY^n_t ) } \sigma_t \right], \notag
 \intertext{and using \eqref{eqConvergenceOfFlows} we get}
    \label{eqYNConvergence}
    Y^n &\to Y^0 \quad \text{ uniformly on } [T-h,T]\ \P-a.s., \\
    Z^n &\to Z^0 \quad \text{ in } H^2_{[T-h,T]}. \notag
  \end{align}


  Let us proceed to the next subinterval.
  To make the rough path disappear in the BSDE, we will use a similar transformation via a flow as above.
  As before we need to control the resulting driver of the transformed BSDE, as well its derivatives. For this reason we have to start the flow anew.
  First, we rewrite the BSDEs for $n\ge1$ as
  \begin{align*}
    Y^n_t = Y^n_{T-h} + \int_t^T f(r,Y^n_r,Z^n_r) dr - \int_t^{T-h} H(X_r, Y^n_r) d\rp^n_r - \int_t^{T-h} Z^n_r dW_r.
  \end{align*}
  Then define the flow $\phi^{n,T-h}$ started at time $T-h$, i.e.
  \begin{align*}
    \phi^{n,T-h}(t, x, y) = y + \int_t^{T-h} H( x, \phi^{n,T-h}(r,x,y) ) d\rp^n(r), \quad t \le T-h.
  \end{align*}

  On $[0,T-h]$ define
  \begin{align*}
    (\tY^{n,T-h}_\cdot,\tZ^{n,T-h}_\cdot)
    &:=
    (
      (\phi^{n,T-h})^{-1}(\cdot, X_\cdot, Y^n_\cdot), \\
      &\qquad
      - \frac{ \partial_x \phi^{n,T-h}( \cdot, X_{\cdot}, (\phi^{n,T-h})^{-1}(\cdot, X_\cdot, Y^n_\cdot) ) }{ \partial_y \phi^{n,T-h}(\cdot,X_\cdot,(\phi^{n,T-h})^{-1}(\cdot, X_\cdot, Y^n_\cdot)) } \sigma_\cdot
      +
      \frac{1}{ \partial_y \phi^{n,T-h}(\cdot, X_\cdot, (\phi^{n,T-h})^{-1}(\cdot, X_\cdot, Y^n_\cdot) ) } Z^n_\cdot
    ).
  \end{align*}

  Then
  \begin{align*}
    \tY^{n,T-h}_t
    =
    Y^n_{T-h}
    +
    \int_t^{T-h} \tilde{f}^{n,T-h}(r, X_r, \tY^{n,T-h}_r, \tZ^{n,T-h}_r) dr
    -
    \int_t^{T-h} \tZ^{n,T-h}_r dW_r,
  \end{align*}
  where
  \begin{align*}
    \tilde{f}^{n,T-h}(t,x,\ty,\tz) &:=
    \frac{1}{\partial_y \phi^{n,T-h}}
        \Bigl\{  
        f\left(t, \phi^{n,T-h}, \partial_y \phi^{n,T-h} \tz + \partial_x \phi^{n,T-h} \sigma_t \right)
          +
          \langle \partial_x \phi^{n,T-h}, \itob_t \rangle \\
          &\qquad \qquad
          +
          \frac{1}{2} \Tr\left[ \partial_{xx} \phi^{n,T-h} \sigma_t \sigma_t^T \right]
          +
          \langle \tz, \left( \partial_{xy} \phi^{n,T-h} \sigma_t \right)^T  \rangle
          +
          \frac{1}{2} \partial_{yy} \phi^{n,T-h} |\tz|^2
        \Bigr\}.
  \end{align*}

  This BSDE is also defined for $n=0$ and as before we get
  via Lemma \ref{lemBSDEPropertiesOfFTilde} for the same $h$ and the same $\tilde{C}_{1,f}$ and $\tCunif$ as before
  (here the explicit dependence of these constants is crucial), that on $[T - 2h, T-h]$ we have
  \begin{align*}
    \partial_y \tilde{f}^{0,T-h}(t,x,y,z) \le \tCunif + \delta(\tilde{C}_{1,f},M) |z|^2.
  \end{align*}

  Hence we can apply Comparison Theorem \ref{thmComparison} to get uniqueness of our solution $(\tY^{0,T-h}, \tZ^{0,T-h})$ on
  $[T - 2 h,T - h]$.
  Now, also note that for the terminal value we have from \eqref{eqYNConvergence} and \eqref{eqBoundOnYn}
  \begin{align*}
    &Y^n_{T-h} \to Y^0_{T-h} \quad \P-a.s., \\
    &|Y^n_{T-h}| \le M, \quad n\ge 1.
  \end{align*}

  Hence, again by the argument of Theorem 2.8 in \cite{bibKobylanski}
  \footnote{
  Note that Theorem 2.8 in \cite{bibKobylanski} demands convergence in $L^{\infty}$ of the terminal value. A closer look at the proof though, reveals that
  $\P$-a.s. convergence combined with a uniform deterministic bound ($M$ in our case) is enough.
  To be specific: the convergence of the terminal value is only used at two instances for Theorem 2.8
  and this is in the proof of Proposition 2.4 (which is the main ingredient for Theorem 2.8).
  Firstly, it is used on p.~568, right before Step 2 where it reads ``By Lebesgue's dominated \dots''.
  Secondly, it is used on p.~570, before the end of the proof where it reads ``from which we deduce that \dots''.
  In both cases, the above stated requirement is enough.
  }
  \begin{align*}
    &\tY^{n,T-h} \to \tY^{0,T-h} \quad \text{ uniformly on } [T-2h,T-h]\ \P-a.s., \\
    &\tZ^{n,T-h} \to \tZ^{0,T-h} \quad \text{ in } H^2_{[T-2h,T-h]}.
  \end{align*}

  Finally, reversing the transformation, we get as above
  \begin{align*}
    &Y^n \to Y^0 \quad \text{ uniformly on } [T-2h,T-h]\ \P-a.s., \\
    &Z^n \to Z^0 \quad \text{ in } H^2_{[T-2h,T-h]}.
  \end{align*}

  Then, we can iterate this procedure on suberintervals of length $h$ up to time $0$.
  Without loss of generality we can assume that $T = N h$ for an $N \in \N$.
  Then, patching the results together we get
  \begin{align*}
    \sup_{t\le T} |Y^n_t - Y^0_t|
    \le
    \sum_{k=1}^N \sup_{(k-1) h \le t \le k h} |Y^n_t - Y^0_t|
    \to 0 \qquad \P-a.s.
  \end{align*}
  and
  \begin{align*}
    \E\left[ \int_0^T |Z^n_r - Z^0_r|^2 dr \right]
    =
    \sum_{k=1}^N 
    \E\left[ \int_{(k-1)h}^{k h} |Z^n_r - Z^0_r|^2 dr \right]
    \to 0.
  \end{align*}

  2. \textbf{Uniqueness}
  \newcommand{\rpII}{\bar{\zeta}}

  Let $\rpII^n, n\ge 1$ be another sequence of smooth paths that converges to $\mathbf{\rp}$ in $p$-variation.
  Let $(\bar{Y}^n, \bar{Z}^n)$ be the solutions to BSDEs with data $(\xi, f, H, \rpII^n)$.
  Then, as above
  \begin{align*}
    &\tilde{\bar{Y}}^n \to \tilde{Y}^0 \quad \text{ uniformly on } [T-h,T]\ \P-a.s., \notag \\
    &\tilde{\bar{Z}}^n \to \tilde{Z}^0 \quad \text{ in } H^2_{[T-h,T]}.
  \end{align*}

  And hence
  \begin{align*}
    &\bar{Y}^n \to Y^0 \quad \text{ uniformly on } [T-h,T]\ \P-a.s., \notag \\
    &\bar{Z}^n \to Z^0 \quad \text{ in } H^2_{[T-h,T]}.
  \end{align*}

  Note that the choice of $h$ in the proof of existence only depended on properties of the limiting function $\tilde{f}^0$,
  so we can use the same value here.
  One can now iterate this argument up to time $0$ to get
  \begin{align*}
    &\bar{Y}^n \to Y^0 \quad \text{ uniformly on } [0,T]\ \P-a.s., \notag \\
    &\bar{Z}^n \to Z^0 \quad \text{ in } H^2_{[0,T]},
  \end{align*}
  as desired.
  
  3. \textbf{Continuity of the solution map}

  We note that for a given $B>0$, all terminal values $\xi$ such that $|\xi| \le B$
  and all geometric $p$-rough paths with $||\mathbf{\rp}||_{\pvar;[0,T]} \le B$ we can choose an $h = h(B) >0$ such that
  the above constructed unique solution $(Y^0,Z^0)$ to the BSDE
  \eqref{eqBSDEUntransformedRough}
  is given by
  \begin{align*}
    Y^0_t
    &=
    \begin{cases}
    &\phi^{0,T}(t, X_t, \tY^T_t), \qquad t \in [T-h,T],\\
    &\phi^{0,T-h}(t, X_t, \tY^{T-h}_t), \qquad t \in [T-2h,T-h],\\
    &\dots \\
    &\phi^{0,h}(t, X_t, \tY^{h}_t), \qquad t \in [0,h],
    \end{cases} \\
    Z^0_t
    &=
    \begin{cases}
      &\partial_y \phi^{0,T}(t, X_t, \tY^{0,T}_t ) \left[ \tZ^{0,T}_t + \frac{ \partial_x \phi^{0,T}(t, X_t, \tY^{0,T}_t) }{ \partial_y \phi^0(t, X_t, \tY^{0,T}_t) } \sigma_t \right], \quad t \in [T-h,T], \\
      &\partial_y \phi^{0,T-h}(t, X_t, \tY^{0,T-h}_t ) \left[ \tZ^{0,T-h}_t + \frac{ \partial_x \phi^{0,T-h}(t, X_t, \tY^{0,T-h}_t) }{ \partial_y \phi^{0,T-h}(t, X_t, \tY^{0,T-h}_t) } \sigma_t \right], \quad t \in [T-2h,T-h], \\
      &\dots \\
      &\partial_y \phi^{0,h}(t, X_t, \tY^{0,h}_t ) \left[ \tZ^{0,h}_t + \frac{ \partial_x \phi^{0,h}(t, X_t, \tY^{0,h}_t) }{ \partial_y \phi^{0,h}(t, X_t, \tY^{0,h}_t) } \sigma_t \right], \quad t \in [0,h], \\
    \end{cases}
  \end{align*}
  where we used the unique solutions to the following BSDEs
  \begin{align*}
    \tY^{0,T}_t &= \xi + \int_t^T \tilde{f}^{0,T}(r,X_r,\tY^{0,T}_r, \tZ^{0,T}_r) dr - \int_t^T \tZ^{0,T}_r dW_r, \\
    \tY^{0,T-h}_t &= \phi^{0,T}(T-h,X_{T-h}, \tY^{0,T}_{T-h}) + \int_t^{T-h} \tilde{f}^{0,T-h}(r,X_r,\tY^{0,T-h}_r, \tZ^{0,T-h}_r) dr - \int_t^{T-h} \tZ^{0,T-h}_r dW_r, \\
    \dots & \\
    \tY^{0,h}_t &= \phi^{0,2h}(h,X_{h}, \tY^{0,2h}_{h}) + \int_t^{h} \tilde{f}^{0,h}(r,X_r,\tY^{0,h}_r, \tZ^{0,h}_r) dr - \int_t^{h} \tZ^{0,h}_r dW_r.
  \end{align*}

  From this representation and
  stability results on BSDEs (Theorem 2.8 in \cite{bibKobylanski})
  it easily follows that the solution map
  \begin{align*}
    C^{\pvar}([0,T], G^{[p]}(\R^d)) \times L^{\infty}(\F_T) \to H^{\infty}_{[0,T]} \times H^2_{[0,T]}
  \end{align*}
  is continuous in balls of radius $B$.
  Since this is true for every $B>0$ we get the desired result.
\end{proof}

\section{The Markovian Setting - Connection To Rough PDEs}
\label{sectTheMarkovianSetting}

We now specialize to a Markovian model.
We are interested in solving the following forward backward stochastic differential equation for $(t_0,x_0) \in [0,T] \times \R^n$
\begin{align}
  \label{eqFBSDE}
  X^{t_0,x_0}_t
  &=
  x + \int_{t_0}^t \sigma(r,X^{t_0,x_0}_r) dW_r + \int_{t_0}^t b(r, X^{t_0,x_0}_r) dr, \quad t \in [{t_0},T], \notag \\
  Y^{t_0,x_0}_t
  &=
  g(X^{t_0,x_0}_T)
  +
  \int_t^T f(r, X^{t_0,x_0}_r, Y^{t_0,x_0}_r, Z^{t_0,x_0}_r) dr \\
  &\qquad
  +
  \int_t^T H(X^{t_0,x_0}_r, Y^{t_0,x_0}_r) d\mathbf{\rp}_r
  -
  \int_t^T Z^{t_0,x_0}_r dW_r, \quad t \in [{t_0},T]. \notag
\end{align}
Here
$\sigma: [0,T] \times \R^n \to \R^{n\times m}, b: [0,T] \times \R^n \to \R^n, f: [0,T] \times \R^n \times \R \times \R^m$ are
continuous mappings.

Assume for the moment that $\mathbf{\rp}$ is actually a smooth path. Then this is connected to the PDE
\begin{align}
  \label{eqPDESmooth}
  &\partial_t u(t,x) + \frac{1}{2} \Tr[ \sigma(t,x) \sigma(t,x)^T D^2 u(t,x) ] + \langle b(t,x), Du(t,x) \rangle \notag \\
  &\qquad + f(t,x, u(t,x), D u(t,x) \sigma(t,x) ) + H(x,u(t,x)) \dot{\rp}_t = 0, \quad t \in [0,T), x \in \R^n, \\
  &u(T,x) = g(x), \quad x \in \R^n. \notag
\end{align}

We will make this connection explicit after introducing the following adaption (and strengthening)
of previous assumptions for the Markovian setting:

\begin{itemize}
  \item[(MA1)]
    There exists a constant $C_\sigma > 0$ such that for $(t,x) \in [0,T] \times \R^n$
    \begin{align*}
      |\sigma(t,x)| &\le C_\sigma, \\
      |\partial_{x_i} \sigma(t,x)| &\le C_\sigma, \quad i=1,\dots,n.
    \end{align*}

  \item[(MA2)]
    There exists a constant $C_b > 0$ such that for $(t,x) \in [0,T] \times \R^n$
    \begin{align*}
      |b(t,x)| &\le C_b, \\
      |\partial_x b(t,x)| &\le C_b.
    \end{align*}


  \item[(MF1)]
    There exists a constant $C_{1,f} > 0$ such that for $(t,x,\uu,\pp) \in [0,T] \times \R^n \times \R \times \R^n$
    \begin{align*}
      &|f(t,x,\uu,\pp)| \le C_{1,f}, \\ 
      &|\partial_\pp f(t,x,\uu,\pp)| \le C_{1,f}. 
    \end{align*}

  \item[(MF2)]
    There exists a constant $C_{2,f} > 0$ such that such that for $(t,x,\uu,\pp) \in [0,T] \times \R^n \times \R \times \R^n$
    \begin{align*}
      \partial_\uu f(t,x,\uu,\pp) \le C_{2,f}.
    \end{align*}

  \item[(MF3)]
    There exists a constant $C_{3,f} > 0$ such that such that for $(t,x,\uu,\pp) \in [0,T] \times \R^n \times \R \times \R^n$
    \begin{align*}
      \partial_x f(t,x,\uu,\pp) \le C_{3,f} + C_{3,f} |\pp|^2,
    \end{align*}
    and $f$ is uniformly continuous in $x$, uniformly in $(t,y,z)$.

  \item[(MG1)]
  $g$ is bounded and uniformly continuous.

\end{itemize}

We again consider for a smooth (or rough) path $\rp$ the flow
\begin{align}
  \label{eqFlowRedundant}
  \phi(t,x,y) = y + \int_t^T \sum_{k=1}^d H_k(x, \phi(r,x,y)) d \rp^k(r).
\end{align}

In what follows $BUC([0,T] \times \R^n)$ (resp. $BUC(\R^n)$) denotes the space of bounded uniformly continuous functions on $[0,T] \times \R^n$ (resp. $\R^n$) with the topology of uniform convergence on compacta.
\begin{proposition}
  \label{propBSDEToViscositySolution}
  Assume (MA1), (MA2), (MF1), (MF2), (MF3), (MG1) and let $H$ be Lipschitz on $\R^n \times \R$.
  For every $(t_0,x_0) \in [0,T]\times \R^n$ let $(Y^{t_0,x_0},Z^{t_0,x_0})$ be the solution to \eqref{eqFBSDE}
  Then $u(t,x) := Y^{t,x}_t$ is a viscosity solution to \eqref{eqPDESmooth} in $BUC([0,T], \R^n)$.
  It is the only viscosity solution in this space.
\end{proposition}
\begin{proof}
  The fact that $u$ is a bounded, uniformly continuous viscosity solution follows
  from Proposition 2.5 in \cite{bibBarlesBuckdahnPardoux}.
  Uniqueness of a viscosity solution to \eqref{eqPDESmooth} follows from Theorem \ref{thmComparisonForParabolic}.
\end{proof}

Let now $p\ge1$, $\rp^n, n=1,2,\dots,$ be smooth paths in $\R^d$ and $\gamma > p$.
Assume $\rp^n \to \mathbf{\rp}^0$
in $p$-variation, for a $\mathbf{\rp}^0 \in C^{\pvar}( [0,T], G^{[p]}(\R^d) )$.
Assume (MA1), (MA2), (MF1), (MF2), (MF3), (MG1) and ($H_{p,\gamma}$), 
so that especially Theorem \ref{thmBSDEConvergence} holds true.
It follows that the corresponding $u^n$ (as given in Theorem \ref{propBSDEToViscositySolution}) converge
\textit{pointwise} to some function $u^0$, i.e.
\begin{align*}
  u^n(t,x) \to u^0(t,x) \qquad t \in [0,T], x \in \R^n.
\end{align*}

Again, the limiting function $u^0$ does not depend on the approximating sequence, but only on the limiting rough path $\mathbf{\rp}^0$.
We could hence define this $u^0$ to be the solution solution to \eqref{eqPDESmooth}.
But it is not straightforward, via this approach, to show uniform convergenc on compacta as well as continuity of the solution map.
We hence work directly on the PDEs, as in \cite{bibCaruanaFrizOberhauser} and \cite{bibFrizOberhauser}.
First we get the respective versions of Lemma \ref{lemBSDETransformation} and Lemma \ref{lemBSDEPropertiesOfFTilde}.

\begin{lemma}
  \label{lemPDETransformation}
  Assume (MA1), (MA2), (MF1), (MF2), (MG1) and let $H(x,\cdot) = (H_1(x,\cdot), \dots, H_d(x,\cdot))$ be a collection of Lipschitz vector fields on $\R$.
  Let a smooth path $\rp$ be given.
  Let $u$ be the unique viscosity solution to \eqref{eqPDESmooth}.

  Then $v(t,x) := \phi^{-1}(t,x,u(t,x))$ is a viscosity solution to
  \begin{align*}
    &\partial_t v(t,x) + \frac{1}{2} \Tr[ \sigma(t,x) \sigma(t,x)^T D^2 v(t,x) ] + \langle b(t,x), Dv(t,x) \rangle \notag \\
    &\qquad + \tilde{f}(t,x, v(t,x), D v(t,x) \sigma(t,x)) = 0, \quad t \in [0,T), x \in \R^n, \\
    &v(T,x) = g(x), \quad x \in \R^n, \notag
  \end{align*}
  where (in what follows the $\phi$ will always be evaluated at $(t,x,\tu)$)
  \begin{align*}
    \tilde{f}(t,x,\tu,\tp)
    &= 
    \frac{1}{\partial_y \phi}
    \Bigl\{  
      f\left(t, \phi, \partial_y \phi \tp + \partial_x \phi \sigma(t,x) \right)
      +
      \langle \partial_x \phi, b(t,x) \rangle
      +
      \frac{1}{2} \Tr\left[ \partial_{xx} \phi \sigma(t,x) \sigma(t,x)^T \right] \\
      &\qquad \qquad
      +
      \langle \tp, \left( \partial_{xy} \phi \sigma(t,x) \right)^T  \rangle
      +
      \frac{1}{2} \partial_{yy} \phi |\tp|^2
    \Bigr\}.
  \end{align*}
\end{lemma}
\begin{proof}
  This is an application of Lemma 5 in \cite{bibFrizOberhauser}.
\end{proof}

\begin{lemma}
  \label{lemPDEPropertiesOfFTilde}
  Let $p\ge 1$, $\mathbf{\rp} \in C^{\pvar}( [0,T], G^{[p]}(\R^d) )$ and $\gamma > p$.
  Assume (MA0), (MA1), (MA2), (MF1), (MF2), (MF3), (G1) and ($H_{p,\gamma}$).
  Let $\phi$ be the flow corresponding to equation \eqref{eqFlowRedundant} (solved as a rough differential equation).
  Then
  \begin{align*}
    \tilde{f}(t,x,\tu,\tp)
    &= 
    \frac{1}{\partial_y \phi}
    \Bigl\{  
      f\left(t, \phi, \partial_y \phi \tp + \partial_x \phi \sigma(t,x) \right)
      +
      \langle \partial_x \phi, b(t,x) \rangle
      +
      \frac{1}{2} \Tr\left[ \partial_{xx} \phi \sigma(t,x) \sigma(t,x)^T \right] \\
      &\qquad \qquad
      +
      \langle \tp, \left( \partial_{xy} \phi \sigma(t,x) \right)^T  \rangle
      +
      \frac{1}{2} \partial_{yy} \phi |\tp|^2
    \Bigr\}
  \end{align*}
  satisfies:
  \begin{itemize}
    \item There exists a constant $\tilde{C}_{1,f} > 0$ depending only on
      $C_\sigma$, $C_\itob$,
      $C_{1,f}$,
      $C_H$ and
      $||\mathbf{\rp}||_{\pvar;[0,T]}$
      such that for $(t,x,\tu,\tp) \in [0,T] \times \R^n \times \R \times \R^n$
      \begin{align*}
        |\tilde{f}(t,x,\tu,\tp)| &\le \tilde{C}_{1,f} + \tilde{C}_{1,f} |\tp|^2, \\
        |\partial_\tp \tilde{f}(t,x,\tu,\tp)| &\le \tilde{C}_{1,f} + \tilde{C}_{1,f} |\tp|. 
      \end{align*}
    \item
      There exists a constant $\tCunif > 0$ that only depends on
      $C_\sigma$, $C_\itob$,
      $C_{2,f}$,
      $C_H$ and
      $||\mathbf{\rp}||_{\pvar;[0,T]}$
      such that
      for every $\varepsilon > 0$ there exists an $h_\varepsilon>0$
      that only depends on
      $C_\sigma$, $C_\itob$,
      $C_H$ and
      $||\mathbf{\rp}||_{\pvar;[0,T]}$
      such that for $(t,x,\tu,\tp) \in [T-h_\varepsilon,T] \times \R^n \times \R \times \R^n$
      \begin{align*}
        \partial_\tu \tilde{f}(t,x,\tu,\tp) \le \tCunif + \varepsilon |\tp|^2.
      \end{align*}
    \item
      There exists a $\tilde{C}_{3,f} > 0$ that only depends on 
      $C_\sigma$, $C_\itob$,
      $C_{2,f}$,
      $C_{3,f}$,
      $C_H$ and
      $||\mathbf{\rp}||_{\pvar;[0,T]}$
      such that for $(t,x,\tu,\tp) \in [0,T] \times \R^n \times \R \times \R^n$
      \begin{align*}
        \partial_x \tilde{f}(t,x,\tu,\tp) \le \tilde{C}_{3,f} + \tilde{C}_{3,f} |\tp|^2.
      \end{align*}
  \end{itemize}
\end{lemma}
\begin{proof}
  The first three inequalities follow as in Lemma \ref{lemBSDEPropertiesOfFTilde}.
  Now for $i \le n$ we have
  \begin{align*}
    &\partial_{x_i} \tilde{f}(t,x,\tu,\tp)\\
    &=
      - \partial_{x_i \uu} \phi \frac{1}{ \partial_\uu \phi } \tilde{f}(t,x,\tu,\tp) \\
      &\quad
      + \frac{1}{\partial_\uu \phi}
      \Bigl[
        \partial_\uu f( t,x,\phi, \partial_\uu \phi \tp + \partial_x \phi \sigma(t,x) ) \partial_{x_i} \phi \\
        &\qquad
        +
        \partial_\pp f( t,x,\phi, \partial_\uu \phi \tp + \partial_x \phi \sigma(t,x) )
        \left(
          \partial_{x_i \uu} \phi \tp
          +
          \partial_{x_i x} \phi \sigma(t,x) + \partial_x \phi \partial_{x_i} \sigma(t,x)
        \right)^T \\
        &\qquad
        +
        \langle \partial_{x_i x} \phi, b(t,x) \rangle
        +
        \langle \partial_x \phi, \partial_{x_i} b(t,x) \rangle \\
        &\qquad
        +
        \frac{1}{2} \Tr\left[ \partial_{x_i xx} \phi \sigma(t,x) \sigma(t,x)^T \right]
        +
        \frac{1}{2} \Tr\left[ \partial_{xx} \phi \partial_{x_i} \sigma(t,x) \sigma(t,x)^T \right]
        +
        \frac{1}{2} \Tr\left[ \partial_{xx} \phi \sigma(t,x) \partial_{x_i} \sigma(t,x)^T \right] \\
        &\qquad
        +
        \langle \tp, \left( \partial_{x_i x\uu} \phi \sigma(t,x) \right)^T  \rangle
        +
        \langle \tp, \left( \partial_{x\uu} \phi \partial_{x_i} \sigma(t,x) \right)^T  \rangle
        +
        \frac{1}{2} \partial_{x_i \uu\uu} \phi |\tp|^2
      \Bigr].
  \end{align*}
  So
  \begin{align*}
    &|\partial_{x_i} \tilde{f}(t,x,\tu,\tp)|\\
    &\le | \partial_{x_i \uu} \phi|  |\frac{1}{ \partial_\uu \phi }| |\tilde{f}(t,x,\tu,\tp)| \\
      &\quad
      + |\frac{1}{\partial_\uu \phi}|
      \Bigl[
        |\partial_\uu f( t,x,\phi, \partial_\uu \phi \tp + \partial_x \phi \sigma(t,x) )| |\partial_{x_i} \phi| \\
        &\qquad
        +
        |\partial_\pp f( t,x,\phi, \partial_\uu \phi \tp + \partial_x \phi \sigma(t,x) )|
        \left(
          |\partial_{x_i \uu} \phi| |\tp|
          +
          |\partial_{x_i x} \phi| |\sigma(t,x)| + |\partial_x \phi| |\partial_{x_i} \sigma(t,x)|
        \right) \\
        &\qquad
        +
        |\langle \partial_{x_i x} \phi| |b(t,x)|
        +
        |\partial_x \phi| |\partial_{x_i} b(t,x)| \\
        &\qquad
        +
        \frac{1}{2} |\partial_{x_i xx} \phi| |\sigma(t,x)|^2
        +
        |\partial_{xx} \phi| |\partial_{x_i} \sigma(t,x)| |\sigma(t,x)| \\
        &\qquad
        +
        |\tp| |\partial_{x_i x\uu} \phi| |\sigma(t,x)|
        +
        |\tp| |\partial_{x\uu} \phi| |\partial_{x_i} \sigma(t,x)|
        +
        \frac{1}{2} |\partial_{x_i \uu\uu} \phi| |\tp|^2
      \Bigr]
    | \\
    &\le
    \tilde{C}_{3,f} + \tilde{C}_{3,f} |\tp|^2
  \end{align*}
  with a constant $\tilde{C}_{3,f}$ only depending on
  $C_\sigma$, $C_\itob$,
  $C_{2,f}$,
  $C_{3,f}$,
  $C_H$ and
  $||\mathbf{\rp}||_{\pvar;[0,T]}$.
  Here we have used the first inequality of the statement to bound $\tilde{f}$,
  (F1), (F2) to bound the $\uu$ and $\pp$ derivative of $f$
  and Lemma \ref{lemFlowEstimates} to bound the flow and its derivatives.

  Now summing over $i$ we get the desired result.
\end{proof}

\begin{theorem}
  Let $p\ge1$, $\gamma > p$ and
  let $\rp^n, n=1,2,\dots$ be smooth paths in $\R^d$.
  Assume
  \begin{align*}
    \rp^n \to \mathbf{\rp}
  \end{align*}
  in $p$-variation, for a $\mathbf{\rp} \in C^{\pvar}( [0,T], G^{[p]}(\R^d) )$.
  Assume (MA1), (MA2), (MF1), (MF2), (MF3), (MG1) and ($H_{p,\gamma}$).
  Let $u^n \in BUC([0,T] \times \R^n)$ be the solution to \eqref{eqPDESmooth} with driving path $\rp^n$ (Theorem \ref{propBSDEToViscositySolution}).
  Then there exists a $u \in BUC([0,T] \times \R^n)$, only dependent on $\mathbf{\rp}$ but not on the approximating sequence $\rp^n$, such that
  \begin{align*}
    u^n \to u \qquad \text{ locally uniformly}. 
  \end{align*}F
  We write (formally)
  \begin{align}
    \label{eqPDERough}
    &du +
      \left[ \frac{1}{2} \Tr[ \sigma(t,x) \sigma(t,x)^T D^2 u(t,x) ] + \langle b(t,x), Du(t,x) \rangle
             + f(t,x, u(t,x), D u(t,x) \sigma(t,x) )
      \right] dt \notag \\
    &\qquad
      + H(x,u(t,x)) d \mathbf{\rp}(t) = 0, \quad t \in (0,T), x \in \R^n, \\
    &u(T,x) = g(x), \quad x \in \R^n. \notag
  \end{align}

  Furthermore, the solution map
  \begin{align*}
    C^{\pvar}([0,T], G^{[p]}(\R^d)) \times BUC(\R^n) &\to BUC( [0,T] \times \R^n), \\
    (\mathbf{\rp}, g) &\mapsto u
  \end{align*}
  is continuous.
\end{theorem}
\begin{remark}
  Equations like \eqref{eqPDERough} have been considered in \cite{bibFrizOberhauser}.
  The setting there is more general in the sense that the vector field in front of the rough path
  is allowed to also depend on the gradient.
  On the other hand, $f$ is independent of the gradient and $H$ is linear.

  For the proof we apply the same ideas as in the proof of Theorem 1 in \cite{bibCaruanaFrizOberhauser}.
  Since comparison on the entire interval $[0,T] $is a subtle issue,
  we mimick our analyis of the BSDE case (Theorem \ref{thmBSDEConvergence}) and proceed on
  small intervals; a similar remark was made in Lions-Souganidis \cite{bibLionsSouganidis}.
\end{remark}
\begin{proof}  For the sake of unified notation, the (rough PDE) solution $u$ will be written as $u^0$ in what follows; similarly, the involved rough path $\mathbf{\rp}$ will be written as  $\mathbf{\rp}^0$.

  1. \textbf{Existence}

  Let $\phi^n, n\ge 0$ be the (ODE, resp. RDE when $n=0$) solution flow
  \begin{align*}
    \phi^n(t,x,\uu) = \uu + \int_t^T H(x, \phi^n(r,x,\uu)) d\rp^n(r).
  \end{align*}
  
  Then, by Lemma \ref{lemPDETransformation}, for $n\ge 1$, $u^n$ is a solution to \eqref{eqPDESmooth} if and only if
  $v^n(t,x) := (\phi^n)^{-1}(t,x,u^n(t,x))$ is a solution to
  \begin{align}
    \label{eqPDETransformedI}
    &\partial_t v^n(t,x) + \frac{1}{2} \Tr[ \sigma(t,x) \sigma(t,x)^T D^2 v^n(t,x) ] + \langle b(t,x), Dv^n(t,x) \rangle \notag \\
    &\qquad + \tilde{f}^n(t,x, v^n(t,x), D v^n(t,x) \sigma(t,x) ) = 0, \quad t \in (0,T), x \in \R^n, \\
    &v^n(T,x) = g(x), \quad x \in \R^n, \notag
  \end{align}
  where
  \begin{align*}
    \tilde{f}^n(t,x,\tu,\tp)
    &= 
    \frac{1}{\partial_\uu \phi^n}
    \Bigl\{
      f\left(t, \phi^n, \partial_\uu \phi^n \tp + \partial_x \phi^n \sigma(t,x) \right)
      +
      \langle \partial_x \phi^n, b(t,x) \rangle
      +
      \frac{1}{2} \Tr\left[ \partial_{xx} \phi^n \sigma(t,x) \sigma(t,x)^T \right] \\
      &\qquad \qquad
      +
      \langle \tp, \left( \partial_{x\uu} \phi^n \sigma(t,x) \right)^T  \rangle
      +
      \frac{1}{2} \partial_{\uu\uu} \phi^n |\tp|^2
    \Bigr\}.
  \end{align*}

  In the proof of Theorem \ref{thmBSDEConvergence} we have already seen that
  $\tilde{f}^n \to \tilde{f}^0$, locally uniformly.
  From the method of semi-relaxed limits (Lemma 6.1, Remark 6.2-6.4 in \cite{bibCrandallIshiiLions}), the pointwise (relaxed) limits
  \begin{align*}
    \bar{v}^0 := \limsup^* v^n, \quad \underbar{v}^0 := \liminf_* v^n,
  \end{align*}
  are viscosity (sub resp. super) solutions to \eqref{eqPDETransformedI} with $n=0$.
  Here we have used the fact, that $\bar{v}^0$ and $\underbar{v}^0$ are indeed finite,
  say bounded in norm by $M>0$.
  This follows from the Feyman-Kac representation (Theorem \ref{propBSDEToViscositySolution}) for each $u^n$, in
  combination with bounds (uniform in $(t_0,x_0)$ and $n$) on the corresponding BSDEs (Corollary 2.2 in \cite{bibKobylanski}). (Although
  not completely obvious, such uniform bounds can also be obtained without BSDE arguments; one would need to exploit comparison for \eqref{eqPDESmooth}, and then \eqref{eqPDETransformedI}, clearly valid when $n \ge 1$, with rough path estimates for RDE solutions which will serve as sub- and super-solutions without spatial structure.)

  By Lemma \ref{lemPDEPropertiesOfFTilde} the function $\tilde{f}^0$ satisfies the conditions of Theorem \ref{thmComparisonForPDE}.
  Hence the PDE \eqref{eqPDETransformedI} for $n=0$ satisfies comparison on $[T-h,T]$ for $h$ sufficiently small, and
  $h$ only depends on $M$ and the constants $\tCunif$, $\tilde{C}_{1,f}$ and $\tilde{C}_{2,f}$ for $\tilde{f}^0$ given
  by Lemma \ref{lemPDEPropertiesOfFTilde}.
  So $v^0(t,x) := \bar{v}^0(t,x) = \underbar{v}^0(t,x), t \in [T-h,T]$ is the unique
  (and continuous, since $\bar{v}, \underbar{v}$ are respectively upper resp. lower semi-continuous) solution
  to \eqref{eqPDETransformedI} with $n=0$ on $[T-h,T]$.
  Moreover, using a Dini-type argument (Remark 6.4 in \cite{bibCrandallIshiiLions}), one sees that this limit must be uniform on compact sets.
  Undoing the transformation, we see that $u^n \to u^0$ locally uniformly on $[T-h,T]$, where
  $u^0(t,x) := \phi^0(t,x,v^0(t,x)),\ t \in [T-h,T]$.

  We proceed to the next subinterval. We use the same argument as above, we just work with a different transformation.
  For $n\ge 0$ let $\phi^{n,T-h}$ be the solution flow started at time $T-h$, i.e.
  \begin{align*}
    \phi^{n,T-h}(t,x,\uu) = \uu + \int_t^{T-h} H(x, \phi^{n,T-h}(r,x,\uu)) d\rp^n(r).
  \end{align*}
  
  Then, for $n\ge 1$, $u^n|_{[0,T-h]}$ is a solution to
  \begin{align*}
    &\partial_t u^n(t,x)
      + \frac{1}{2} \Tr[ \sigma(t,x) \sigma(t,x)^T D^2 u^n(t,x) ] + \langle b(t,x), Du^n(t,x) \rangle \\
    &\quad + f(t,x, u^n(t,x), D u^n(t,x) \sigma(t,x) ) + H(x,u^n(t,x)) \dot{\rp}_r = 0, \quad t \in [0,T-h], x \in \R^n, \\
    &u(T-h,x) = \phi^n( T-h, x, v^n(T-h,x) ), \quad x \in \R^n.
  \end{align*}
  if and only if
  $v^{n,T-h}(t,x) := (\phi^{n,T-h})^{-1}(t,x,u^n(t,x))$ is a solution to
  \begin{align*}
    &\partial_t v^{n,T-h}(t,x) + \frac{1}{2} \Tr[ \sigma(t,x) \sigma(t,x)^T D^2 v^{n,T-h}(t,x) ] + \langle b(t,x), Dv^{n,T-h}(t,x) \rangle \notag \\
    &\qquad + \tilde{f}^{n,T-h}(t,x, v^{n,T-h}(t,x), \sigma(t,x) D v^{n,T-h}(t,x)) = 0, \quad t \in (0,T-h), x \in \R^n, \\
    &v^{n,T-h}(T,x) = \phi^n(T-h,x,v^n(T-h,x)), \quad x \in \R^n,
  \end{align*}
  where of course $\tilde{f}^{n,T-h}$ is defined as $\tilde{f}^n$ was, with $\phi^n$ replaced by $\phi^{n,T-h}$.

  Now we have already shown that the terminal values of these PDEs converge, e.g.
  \begin{align*}
    \phi^n( T-h, \cdot, v^n(T-h,\cdot) ) \to \phi( T-h, \cdot, v(T-h,\cdot) ), \text{ locally uniformly}.
  \end{align*}
  As before, one also shows that $\tilde{f}^{n,T-h} \to \tilde{f}^{0,T-h}$, locally uniformly.
  By Theorem \ref{thmComparisonForPDE} we again get comparison, now on $[T-2h,T-h]$,
  and hence again via the method of semi-relaxed limits we arrive at
  \footnote
  {
    Lemma 6.1 in \cite{bibCrandallIshiiLions} does not take into account converging terminal values.
    But the result is immediate: the relaxed limit is a sub resp. super solution by Lemma 6.1 and
    their terminal value is exactly the limit of the given converging terminal values.
  }
  \begin{align*}
    v^{n,T-h} \to v^{0,T-h} \quad \text{ locally uniformly on } [T-2h,T-h] \times \R^n.
  \end{align*}
  Hence $u^n \to u^0$ locally uniformly on $[T-2h,T-h]$, where
  $u^0(t,x) = \phi^{0,T-h}(t,x,v^{0,T-h}(t,x))$.
  Iterating this argument up to time $0$ we get
  \begin{align*}
    u^n \to u^0 \quad \text{ locally uniformly on } [0,T] \times \R^n,
  \end{align*}
  where $u^0$ is defined on intervals of length $h$ as above.

  2. \textbf{Uniqueness, Continuity of solution map}

  Uniqueness of the limit and continuity of the solution map now follow by the same arguments as
  in the proof of Theorem \ref{thmBSDEConvergence}, adapted to the PDE setting.
\end{proof}

\section{ Connection To BDSDEs }
\label{sectConnectiontoBDSDEs}
Let $\Omega^1 = C([0,T], \R^d)$, $\Omega^2 = C([0,T], \R^m)$, with the respective Wiener measures $\P^1$, $\P^2$ on them.
Let $\Omega = \Omega^1 \times \Omega^2$, with the product measure $\P := \P^1 \otimes \P^2$.
For $(\omega^1,\omega^2) \in \Omega$ let $B(\omega^1,\omega^2) = \omega^1$ be the coordinate mapping with respect to the first component.
Analogously $W(\omega^1,\omega^2) = \omega^2$ is the coordinate mapping with respect to the second component.
In particular, $B$ is a $d$-dimensional Brownian motion and $W$ is an independent $m$-dimensional Brownian motion. 

Define $\F_t := \F_{t,T}^B \vee \F_{0,t}^W$, where $\F_{t,T}^B := \sigma( B_r : r \in [t,T] ), \F_{0,t}^W := \sigma( W_r : r \in [0,t])$.
Note that $\F$ is not a filtration, since it is neither increasing nor decreasing.
In this setting, Pardoux and Peng \cite{bibPardouxPengBDSDEs} considered backward doubly stochastic differential equations (BDSDEs).
An $\F$-adapted process $(Y,Z)$ is called a solution to the BDSDE
\begin{align}
  \label{eqBDSDE}
  Y_t = \xi + \int_t^T f(r,Y_r,Z_r)dr + \int_t^T H(X_r, Y_r) \circ dB_r - \int_t^T Z_r dW_r,
\end{align}
if $\E[ \sup_{t\le T} |Y_t|^2 ] < \infty$, $\E[ \int_0^T |Z_r|^2 dr ] < \infty$ and $(Y,Z)$ satisfies $\P$-a.s. \eqref{eqBDSDE} for $t\le T$.

Under appropriate (essentially Lipschitz) conditions on $f$ and $H$ they were able to show existence and uniqueness of a solution.
\footnote
{
  Pardoux and Peng considered equations, where the Stratonovich integral was actually a backward integral.
  But if $H$ is smooth enough, the formulations are equivalent.
  See also Section 4 in \cite{bibBuckdahnMaI}.
}

The connection to BSDEs with rough driver is given by the following
\begin{theorem}
  Let $p \in (2,3), \gamma > p$.
  Let $\xi \in L^{\infty}(\F_T)$.
  Let $f$ be a random function satisfying (F1) and (F2).
  Moreover, assume (A1), (A2) and ($H_{p,\gamma}$).

  Then by Theorem 1.1 in \cite{bibPardouxPengBDSDEs} there exists a unique solution $(Y,Z)$ to the BDSDE
  \begin{align*}
    Y_t = \xi + \int_t^T f(r,Y_r,Z_r)dr + \int_t^T H(X_r, Y_r) \circ dB_r - \int_t^T Z_r dW_r.
  \end{align*}

  Let $\mathbf{B}_t = \exp( B_t + A_t )$ be the Enhanced Brownian motion (over $B$)
  \footnote{$\mathbf{B}$ is precisely $d$-dimensional Brownian motion enhanced with its iterated integrals in Stratonovich sense; it is in $1-1$ correspondence with Brownian motion enhanced with L\'evy's area; $\exp$ denotes the exponential map from the Lie algebra $\R^d \oplus so(d)$ to the group, realized inside the truncated tensor algebra. See e.g. section 13 in \cite{bibFrizVictoir} for more details.},
  especially $\mathbf{B} \in C^{\pvar}_0( [0,T], G^{2}(\R^d))$ $\P^1\ a.s.$. By setting $\mathbf{B} = 0$ on a null set,
  we get $\mathbf{B} \in C^{\pvar}_0( [0,T], G^{2}(\R^d))$.
  By Theorem \ref{thmBSDEConvergence} we can, for every $\omega^1 \in \Omega^1$,
  construct the solution to the BSDE with rough driver
  \begin{align*}
    Y^{rp}(\omega^1,\cdot)_t
    &= \xi(\cdot) + \int_t^T f(r, Y^{rp}_r, Z^{rp}_r)dr + \int_t^T H(X_r, Y^{rp}(\omega^1,\cdot)) d\mathbf{B}_r(\omega^1) \\
    &\quad- \int_t^T Z^{rp}(\omega^1,\cdot) dW_r(\cdot), \quad t \in [0,T].
  \end{align*}

  We then have for $\P^1-a.e.\ \omega^1$ that $\P^2-a.s.$
  \begin{align*}
    \tilde{Y}_t(\omega^1,\cdot) = \tilde{Y}^{rp}_t(\omega^1,\cdot), \quad t\le T
  \end{align*}
  and  
  \begin{align*}
    &Z_t(\omega^1,\cdot) = Z^{rp}_t(\omega^1,\cdot), \quad dt \otimes \P^2 a.s..
  \end{align*}
\end{theorem}
\begin{proof}
  As in the proof of Theorem \ref{thmBSDEConvergence},
  in the BDSDE setting, one can transform the integral belonging to the Brownian motion $B$ away.
  In \cite{bibBuckdahnMaI} it was shown, that if we let $\phi$ be the stochastic (Stratonovich) flow
  \begin{align*}
    \phi(\omega^1;t,y) = y + \int_t^T H( \phi(\omega^1;r,y) ) \circ dB_r(\omega^1),
  \end{align*}
  then with $\tilde{Y}_t := \phi^{-1}(t, Y_t), \tilde{Z}_t := \frac{1}{\partial_y \phi(t,Y_t)} Z_t$ we have $\P$-a.s.
  \begin{align}
    \label{eqBSDETransformedStochastic}
    \tilde{Y}_t(\omega^1,\omega^2) = \xi(\omega^2) + \int_t^T \tilde{f}(\omega^1,\omega^2;r,\tilde{Y}_r(\omega^1,\omega^2),\tilde{Z}_r(\omega^1,\omega^2)) dr - \int_t^T \tilde{Z}_r(\omega^1,\omega^2) dW_r(\omega^2), \quad t\le T.
  \end{align}
  Here
  \begin{align*}
    \tilde{f}(\omega^1,\omega^2;t,x,\ty,\tz) &:=
      \frac{1}{\partial_y \phi}
      \Bigl\{
        f\left(\omega^2; t, \phi, \partial_y \phi \tz + \partial_x \phi \sigma_t \right)
        +
        \langle \partial_x \phi, \itob_t \rangle
        +
        \frac{1}{2} \Tr\left[ \partial_{xx} \phi \sigma_t \sigma_t^T \right] \\
        &\qquad \qquad
        +
        \langle \tz, \left( \partial_{xy} \phi \sigma_t \right)^T  \rangle
        +
        \frac{1}{2} \partial_{yy} \phi |\tz|^2
        \Bigr\},
  \end{align*}
  where $\phi$ and its derivatives are always evaluated at $(\omega^1;x,\ty)$.
  Especially, by a Fubini type theorem (e.g. Theorem 3.4.1 in \cite{bibBogachev}), there exists $\Omega^1_0$ with $\P^1(\Omega^1_0) = 1$ such that for $\omega^1 \in \Omega^1_0$ equation \eqref{eqBSDETransformedStochastic} holds true $\P^2\ a.s.$.

  On the other hand we can construct $\omega^1$-wise the rough flow
  \begin{align*}
    \phi^{rp}(\omega^1;t,y) = y + \int_t^T H( \phi^{rp}(\omega^1;r,y) ) d\mathbf{B}_r(\omega^1).
  \end{align*}

  Assume for the moment that we have global comparison,
  so that we can solve the transformed BSDE uniquely, i.e.
  for every $\omega^1 \in \Omega^1$, we have
  $\P^2\ a.s.$
  \begin{align*}
    \tilde{Y}^{rp}_t(\omega^1,\omega^2)
    &= \xi(\omega^2) + \int_t^T \tilde{f}^{rp}(\omega^1;r,\tilde{Y}^{rp}_r(\omega^1,\omega^2),\tilde{Z}^{rp}_r(\omega^1,\omega^2)) dr \\
    &\qquad - \int_t^T \tilde{Z}^{rp}_r(\omega^1,\omega^2) dW_r(\omega^2), \quad t \le T,
  \end{align*}
  where
  \begin{align*}
    \tilde{f}^{rp}(\omega^1,\omega^2;t,x,\ty,\tz) &:=
    \frac{1}{\partial_y \phi^{rp}}
      \Bigl\{
      f\left(\omega^2; t, \phi^{rp}, \partial_y \phi^{rp} \tz + \partial_x \phi^{rp} \sigma_t \right)
        +
        \langle \partial_x \phi^{rp}, \itob_t \rangle
        +
        \frac{1}{2} \Tr\left[ \partial_{xx} \phi^{rp} \sigma_t \sigma_t^T \right] \\
        &\qquad \qquad
        +
        \langle \tz, \left( \partial_{xy} \phi^{rp} \sigma_t \right)^T  \rangle
        +
        \frac{1}{2} \partial_{yy} \phi^{rp} |\tz|^2
        \Bigr\},
  \end{align*}
  where $\phi$ and its derivatives are always evaluated at $(\omega^1;x,\ty)$.
  It is a classical rough path result, that there exists $\Omega^1_1$ with  $\P^1(\Omega^1_1) = 1$ such that for $\omega^1 \in \Omega^1_1$, we have
  \begin{align*}
    \phi^{rp}(\omega^1;\cdot,\cdot) = \phi(\omega^1;\cdot,\cdot).
  \end{align*}
  Combining above results we have for $\omega^1 \in \Omega^1_0 \cap \Omega^1_1$ that $\P^2\ a.s.$
  \begin{align*}
    \tilde{Y}_t(\omega^1,\cdot) = \tilde{Y}^{rp}_t(\omega^1,\cdot), t\le T,
  \end{align*}
  and
  \begin{align*}
    &\tilde{Z}_t(\omega^1,\cdot) = \tilde{Z}^{rp}_t(\omega^1,\cdot), \quad dt \otimes \P^2 a.s..
  \end{align*}

  Now, since comparison does \textit{not} necessarily hold globally, we must argue differently.
  Define $A^k := \{ \omega^1 \in \Omega^1 : ||\mathbf{B}(\omega^1)||_{\pvar} \le k \}$.
  Then on $A^k$ we have for an $h = h(k) > 0$ comparison on $[T-h,T]$, and we argue on subsequent intervals as above.
  Now, since $\P( \cup A^k ) = 1$, we get the desired result.
\end{proof}

\appendix

\section{Comparison for BSDEs}
\begin{appendixDefinition}
  Let $\xi \in L^{\infty}(\F_T)$, $W$ an $m$-dimensional Brownian motion
  and $f$ a predictable function on $\Omega \times \R_+ \times \R \times \R^m$.

  We call an adapted process $(Y, Z, C)$ a \textit{supersolution to the BSDE with data $(\xi,f)$} if
  $Y \in H^{\infty}_{[0,T]}$, $Z \in H^2_{[0,T]}$, $C$ is an adapted right continuous increasing process and
  \begin{align*}
    Y_t = \xi + \int_t^T f(r, Y_r, Z_r) dr - \int_t^T Z_r dW_r + \int_t^T dC_r, \qquad t\le T.
  \end{align*}

  We call $(Y, Z, C)$ a \textit{subsolution to the BSDE with data $(\xi,f)$} if
  $(Y,Z,-C)$ is a supersolution.
\end{appendixDefinition}

The following statement as well as its proof are based on Theorem 2.6 in \cite{bibKobylanski}.

\begin{appendixTheorem}
  \label{thmComparison}
  There exists a (universal) strictly positive function
  $\delta: \R_+^2 \to (0,\infty)$ such that the following statement is true.

  Let $(Y^1, Z^1, C^1)$ be a supersolution to the BSDE with
  data $(\xi^1, f^1)$.
  Let $(Y^2, Z^2, C^2)$ be a subsolution to the BSDE with
  data $(\xi^2, f^2)$.
  Let $M \in \R_+$ be a bound for $Y^1$ and $Y^2$, i.e.
  \begin{align*}
    ||Y^1||_{\infty}, ||Y^2||_{\infty} \le M.
  \end{align*}

  Assume that $\P$-a.s.
  \begin{align*}
    f^1(Y^1_t, Z^1_t) &\le f^2(t, Y^1_t, Z^1_t), \quad \forall t \in [0,T], \\
    \xi^1 &\le \xi^2.
  \end{align*}

  Assume that there exist constants $C > 0, L > 0, K > 0$ such that for $(t,y,z) \in [0,T] \times [-M,M] \times \R^m$ 
  \begin{align*}
    &|f^2(t,y,z)| \le L + C |z|^2 \quad \P-a.s., \\
    &|\partial_z f^2(t,y,z)| \le K + C |z| \quad \P-a.s.
  \end{align*}

  Assume that there exists a constant $N > 0$ such that for $(t,y,z) \in [0,T] \times [-M,M] \times \R^m$
  \begin{align}
    \label{eqBoundOnF2y}
    \partial_y f^2(t,y,z) \le N + \delta(C,M) |z|^2 \quad \P-a.s.
  \end{align}

  Then $\P$-a.s.
  \begin{align}
    \label{eqComparison}
    Y^1_t \le Y^2_t, \quad 0 \le t \le T.
  \end{align}
\end{appendixTheorem}
\begin{appendixRemark}
  We note that, as in Theorem 2.6 of \cite{bibKobylanski}, the assumptions could be weakened
  by replacing the constants $L, K, N$
  by deterministic functions $l_t \in L^1(0,T), k_t \in L^2(0,T)$ and $n_t \in L^1(0,T)$.

  In our application of Theorem \ref{thmComparison} in the proof of Theorem \ref{thmBSDEConvergence},
  condition \eqref{eqBoundOnF2y} is not satisfied on $[0,T]$.
  But we are able to choose $h > 0$ small enough, such that it is satisfied on $[T-h,T]$.
  Comparison \eqref{eqComparison} then holds on $[T-h,T]$.
\end{appendixRemark}
\begin{proof}
  1. Let $\lambda > 0, B > 1$ be constants, to be specified later on.
  We begin by constructing several functions, whose good properties we will rely on later in the proof.
  Define
  \begin{align*}
    \gamma(\tilde{y}) := 
    \gamma_{\lambda, B}(\tilde{y}) := 
    \frac{1}{\lambda} \log\left( \frac{ e^{\lambda B \tilde{y}} + 1 }{ B } \right) - M, \quad \tilde{y} \in \R.
  \end{align*}
  
  Then
  \begin{align*}
    \gamma^{-1}(y) = \frac{1}{\lambda B} \log\left( B e^{\lambda (y + M)} - 1 \right), \quad
    \gamma'(\tilde{y}) = B \frac{1}{1 + e^{-\lambda B \tilde{y}}}.
  \end{align*}

  Denote $g(y) := e^{-\lambda (y+M)}$, then $0 < g \le 1, \quad \text{ on } [-M,M]$.
  Define
  \begin{align*}
    w(y) 
    &:= \gamma'(\gamma^{-1}(y))
    = B - g(y).
  \end{align*}
  Then
  \begin{align*}
    w'(y)
    &= \lambda g(y), \quad
    w''(y)
    = -\lambda^2 g(y), \\
    \frac{w'(y)}{w(y)}
    &= \frac{ \lambda g(y) }{ B - g(y) }, 
    \quad
    \frac{w''(y)}{w(y)}
    = \frac{ -\lambda^2 g(y) }{B - g(y)}.
  \end{align*}

  In particular $w > 0$ on $[-M, M]$.

  Define $\alpha(y) := \gamma^{-1}(y)$.
  Then, since $(Y^1, Z^1, C^1)$ is a supersolution to the BSDE with data $(\xi^1, f^1)$, It\^{o} formula gives
  \begin{align*}
    \alpha(Y^1_t) =
    \alpha(Y^1_0)
    - \int_0^t \alpha'(Y^1_r) f^1(r, Y^1_r, Z^1_r) dr
    + \int_0^t \alpha'(Y^1_r) Z^1_r dW_r
    - \int_0^t \alpha'(Y^1_r) dC_r
    + \int_0^t \alpha''(Y^1_r) |Z^1_r|^2 dr.
  \end{align*}

  Define
  \begin{align*}
    \tilde{Y^{1}} := \alpha(Y^{1}), \quad
    \tilde{Z^{1}} := \frac{Z^{1}}{\gamma'(\tilde{Y}^{1})} = \frac{Z^{1}}{w(Y^{1})}.
  \end{align*}
  and
  \begin{align*}
    F^{1}(t, \tilde{y}, \tilde{z}) := \frac{1}{\gamma'(\tilde{y})} \left[ f^{1}(t, \gamma(\tilde{y}), \gamma'(\tilde{y}) \tilde{z}) + \frac{1}{2} \gamma''(\tilde{y}) |\tilde{z}|^2 \right].
  \end{align*}

  Since $\alpha' > 0$
  we have that $(\tilde{Y}^1, \tilde{Z}^1, \int_0^{\cdot} \alpha'(Y^1_r) dC^1_r)$ is a supersolution
  to the BSDE with data $(\alpha(\xi^1), F^1)$.
  Analogously
  we have that $(\tilde{Y}^2, \tilde{Z}^2, \int_0^{\cdot} \alpha'(Y^2_r) dC^2_r)$ is a subsolution
  to the BSDE with data $(\alpha(\xi^2), F^2)$.
  Since $\alpha$ is increasing, it is now enough to verify that $\tilde{Y}^1 \le \tilde{Y}^2$.

  For that we will verify, that $F^2$ satisfies the conditions of Proposition 2.9 in \cite{bibKobylanski}.
  Especially we will show, that there exists a constant $G > 0$ such that
  \begin{align}
    \label{eqStructureCondition}
    \partial_y f^2(t,y,z) + A |\partial_z f^2(t,y,z)|^2 \le G, \quad \forall (t,y,z) \in [0,T]\times\R\times\R^m.
  \end{align}

  For simplicity denote $F := F^2, f := f^2$.
  Denote $y = \gamma( \tilde{y} ),\ z = \gamma'(\tilde{y}) \tilde{z} = w(y) \tilde{z}$.
  For convenience $w$ and its derivatives will always be evaluated at $y$.
  Then
  \begin{align*}
    \partial_{\tilde{z}} F(t, \tilde{y}, \tilde{z})
    &= \partial_z f(t, y, z) + z \frac{w'}{w}, \\
    \partial_{\tilde{y}} F(t, \tilde{y}, \tilde{z})
    &= \frac{1}{w} \left[ \frac{1}{2} w'' |z|^2 + w' \left( \partial_z f(t,y,z) z - f(t,y,z) \right) \right] + \partial_y f(t,y,z).
  \end{align*}
  
  Hence
  \begin{align*}
    \partial_{\tilde{y}} F(t, \tilde{y}, \tilde{z})
    &\le
    \frac{1}{w} \left[ \frac{1}{2} w'' |z|^2 + w' \bigl( |z| [ K + C |z| ] + L + C|z|^2 \bigr) \right] + \partial_y f(t,y,z)
  \end{align*}
  and
  \begin{align*}
    |\partial_{\tilde{z}} F(t, \tilde{y}, \tilde{z})|^2
    \le \left[ K + C |z| + \frac{w'}{w} |z| \right]^2.
  \end{align*}

  So, for $A > 0$
  \begin{align*}
    (\partial_{\tilde{y}} F + A |\partial_{\tilde{z}} F|^2)(t,\tilde{y}, \tilde{z})
    &\le
    |z|^2 \left[ \frac{1}{2} \frac{w''}{w} + \frac{w'}{w} 2 C + A (C + \frac{w'}{w})^2 \right]
    +
    K |z| \left[ \frac{w'}{w} + 2 A \left( C + \frac{w'}{w} \right) \right] \\
    &\qquad
    +
    \frac{w'}{w} L
    +
    |\partial_y f(t,y,z)|
    +
    A K^2.
  \end{align*}

  Note, that for the second term we have
  \begin{align*}
    K |z| \left[ \frac{w'}{w} + 2 A \left( C + \frac{w'}{w}\right) \right]
    &\le
    K |z| \left[ (1 + 2 A) \left( C + \frac{w'}{w}\right) \right] \\
    &\le
    A \left( C + \frac{w'}{w}\right)^2 |z|^2 +
    \frac{(1 + 2 A)^2}{A} K^2.
  \end{align*}

  Hence
  \begin{align*}
    (\partial_{\tilde{y}} F + A |\partial_{\tilde{z}} F|^2)(t,\tilde{y}, \tilde{z})
    &\le
    |z|^2
    \left[ \frac{1}{2} \frac{w''}{w} + \frac{w'}{w} 2 C + 2 A (C + \frac{w'}{w})^2 \right] \\
    &\quad
    +
    \frac{w'}{w} L
    +
    |\partial_y f(t,y,z)|
    +
    \left( A + \frac{ (1 + 2A)^2 }{ A } \right) K^2.
  \end{align*}

  Now
  \begin{align*}
    \frac{1}{2} \frac{w''}{w} + \frac{w'}{w} 2 C + 2 A ( C + \frac{w'}{w} )^2
    &=
    \frac{1}{2} \frac{w''}{w} + \frac{w'}{w} 2 C + 2 A C^2 + 4 A C \frac{w'}{w} + 2 A ( \frac{w'}{w} )^2 \\
    &= 
    -\frac{\lambda^2}{2} \frac{ g(y) }{ B - g(y) }
    +
    2 C (1 + 2 A) \frac{ \lambda g(y) }{ B - g(y) } 
    +
    2 A C^2 + 2 A \frac{ \lambda^2 g(y)^2 }{ (B - g(y))^2 } \\
    &=
    \frac{g(y)}{ (B - g(y))^2 } \Bigl[ - \frac{ \lambda^2 }{2} ( B - g(y) ) + 2 C (1+2A) \lambda (B-g(y)) \\
    &\qquad \qquad
    +
    2 A \lambda^2 g(y) \Bigr] + 2 A C^2 \\
    &=
    \frac{g(y)}{ (B - g(y))^2 } \left[ \frac{ \lambda^2 }{2} ( (1 + 4 A) g(y) - B ) + 2 C (1+2A) \lambda (B-g(y)) \right] + 2 A C^2.
  \end{align*}

  For all $A < 1$ we hence have
  \begin{align*}
    \frac{1}{2} \frac{w''}{w} + \frac{w'}{w} 2 C + 2 A ( C + \frac{w'}{w} )^2
    &\le
    \frac{g(y)}{ (B - g(y))^2 } \left[ \frac{ \lambda^2 }{2} ( 5 g(y) - B ) + 2 C 3 \lambda (B-g(y)) \right] + 2 A C^2.
  \end{align*}
  Now, choose
  $B = 6$. Hence $5 g(y) - B \le -1,\ y \in [-M,M]$.
  Then choose $\lambda = \lambda(C)$ sufficiently large such that the term in square brackets is strictly negative, say
  smaller then $-1$ for all $y \in [-M,M]$.
  This is possible since it is a polynomial in $\lambda$ and the leading power has a negative coefficient.
  Then for $y \in [-M,M]$
  \begin{align*}
    \frac{g(y)}{ (6 - g(y))^2 } \left[ \frac{ \lambda^2 }{2} ( 5 g(y) - 6 ) + 2 C 3 \lambda (6-g(y)) \right]
    &\le
    - \frac{g(y)}{ (6 - g(y))^2 } \\
    &\le
    - \frac{1}{36} e^{-\lambda 2 M} =: - 2 \delta,
  \end{align*}
  where $\delta$ depends only $M$ and $\lambda$ and hence only on $M$ and $C$, i.e.
  \begin{align*}
    \delta = \delta(C,M) = \frac{1}{72} e^{-\lambda(C) 2 M}.
  \end{align*}
  Now choose $A$ small enough such that $2 A C^2 < \delta$.
  If then for some $N>0$ we have
  \begin{align*}
    \partial_y f(t,y,z) \le N + \delta(C,M) |z|^2,
  \end{align*}
  it follows that
  \begin{align*}
    (\partial_{\tilde{y}} F + A |\partial_{\tilde{z}} F|^2)(t,\tilde{y}, \tilde{z})
    &\le
    \frac{w'}{w} L + N + \left( A + \frac{ (1 + 2A)^2 }{ A } \right) K^2 \\
    &\le
    \frac{\lambda}{B-1} L + N + \left( A + \frac{ (1 + 2A)^2 }{ A } \right) K^2 \\
    &=: G.
  \end{align*}

  So we have shown \eqref{eqStructureCondition} and comparison the follows from Proposition 2.9  in \cite{bibKobylanski}.
\end{proof}

\section{Flow properties}

Consider the solution flow $\phi$ to
\begin{align}
  \label{eqFormalFlow}
  \phi(t,x,y) = y + \int_t^T H(x, \phi(r,x,y)) d\rp_r,
\end{align}
where $H$ and $\rp$ will be specified in a moment.
We need to control
\[
  \partial_{y} \phi - 1, \partial_{x} \phi, \partial_{xx} \phi,
  \partial_{xy}\phi, \partial_{yy} \phi, \partial_{yyy} \phi, \partial_{xyy}\phi, \partial_{xxy} \phi
\]%
over a small interval $\left[ T-h,T\right] $. Note that each of the above expressions is $0$ when evaluated at $t=T$. 

\bigskip

\begin{appendixLemma}
  \label{lemFlowEstimates}
Let $p \ge 1$, $\mathbf{\rp} \in C^{\pvar}( [0,T], G^{[p]}(\R^d) )$ and $\gamma > p$.
Assume that \thinspace $H_{i}=H_{i}\left( x,y\right) $
has joint regularity of the form%
\[
\sup_{i=1,\dots ,d}\,\left\vert H_{i}\left( \cdot ,\cdot \right) \right\vert
_{\Lip^{\gamma +2}\left( R^{n+1}\right) }\leq c_{1}
\]%
and%
\[
\left\Vert \mathbf{\zeta }\right\Vert _{p\text{-var;}\left[ 0,T\right] }\leq
c_{2}.
\]%

Then, the solution to \eqref{eqFormalFlow} induces a flow of $C^3$ diffeomorphisms, parametrized by $x\in \R^n$, and 
there exists a positive $L = L \left( c_{1},c_{2}, T \right) $ so that, uniformly over $x\in R^{n},y\in R$ and $%
t\in \left[ 0 , T \right] $%
\[
\max
\left\{
  \partial_{x}\phi,
  \partial_{y} \phi,
  \frac{1}{ \partial_{y} \phi },
  \partial_{xx} \phi,
  \partial_{xy}\phi,
  \partial_{yy}\phi,
  \partial_{yyy} \phi,
  \partial_{xyy} \phi,
  \partial_{xxy} \phi
\right\} < L.
\]

Moreover, for every $\varepsilon >0$ there exists a positive $\delta =\delta
\left( c_{1},c_{2}\right) $ so that, uniformly over $x\in R^{n},y\in R$ and $%
t\in \left[ T-\delta , T \right] $%
\[
\max
\left\{
  \partial_{x} \phi,
  \partial_{y} \phi - 1,
  \partial_{xx} \phi,
  \partial_{xy}\phi,
  \partial_{yy} \phi,
  \partial_{yyy} \phi,
  \partial_{xyy}\phi,
  \partial_{xxy} \phi
\right\} <\varepsilon .
\]
\end{appendixLemma}

\begin{proof}
Consider the extended RDE%
\begin{eqnarray*}
d\xi  &=&0 \\
-d\phi  &=&H\left( \xi ,\phi \right) d\zeta 
\end{eqnarray*}%
with terminal data $\left( \xi _{T},\phi _{T}\right) =\left( x,y\right) $.
The assumption on $\left( H_{i}\right) $ implies that $\left( \xi ,\phi
\right) $ evolves according to a rough differential equation with $%
\Lip^{\gamma +2}$-vector fields. In this case, the ensemble%
\[
\hat{\phi}
=
\left(
  \xi,
  \phi,
  \partial_{x}\phi,
  \partial_{y}\phi,
  \partial_{xx}\phi,
  \partial_{xy}\phi,
  \partial_{yy}\phi,
  \partial_{yyy}\phi,
  \partial_{xyy}\phi,
  \partial_{xxy}\phi
\right) 
\]%
can be seen to be the (unique\footnote{%
This is actual a subtle point since uniquess in general requires $%
\Lip_{loc}^{\gamma }$-regularity. The point is that the RDEs obtain by
differentiating the flow have a special structure so that for the final
level of derivatives only rough integration is need; as is well known, for
this it suffices to have $\Lip_{loc}^{\gamma -1}$ regularity. (cf. Chapter 11 in \cite{bibFrizVictoir})
contains a detailed discussion of this.}, non-explosive) solution
to an RDE along $\Lip_{loc}^{\gamma -1}$ vector fields. Thanks to
non-explosivity we can, for fixed terminal data%
\[
\hat{\phi}_{T}=
\left(
  x,
  y,
  0,
  1,
  0,
  0,
  0,
  0,
  0,
  0
\right),
\]
localize the problem and assume without loss of generality that the above
ensemble is driven along $\Lip^{\gamma -1}$ vector fields. Since we want
estimates that are \textit{uniform} in $x,y$ we make another key
observations: there is no loss of generality in taking $\left( x,y\right)
=\left( 0,0\right) $ provided $H$ is replaced by $H_{x,y}=H\left( x+\cdot
,y+\cdot \right) $. This also shifts the derivaties (evaluated at some $%
\left( x,y\right) $) to derivatives evaluated at $\left( 0,0\right) $. As
announced, we can now safely localize, and assume that the vector fields
required for $\hat{\phi},$ obtain by taking formal $\left( x,y\right) $
derivatives in%
\begin{eqnarray*}
d\xi  &=&0 \\
-d\phi  &=&H\left( \xi ,\phi \right) d\zeta ,
\end{eqnarray*}%
are globally $\Lip^{\gamma -1}$. A basic estimate (Thm 10.14 in \cite{bibFrizVictoir}) for RDE
solutions implies that for some $C=C\left( p,\gamma \right) $%
\[
\left\vert \hat{\phi}_{t}-\hat{\phi}_{T}\right\vert \leq \left\vert \hat{\phi%
}\right\vert _{p\text{-var;}\left[ t,T\right] }=C\times \varphi _{p}\left(
\left\vert H_{x,y}\right\vert _{\Lip^{\gamma +2}}\left\Vert \mathbf{\zeta }%
\right\Vert _{p\text{-var;}\left[ T-h,T\right] }\right),
\]%
where $\varphi _{p}\left( x\right) =\max (x,x^{p}\,)$.At last, we note that $%
\left\vert H_{x,y}\right\vert _{\Lip^{\gamma +2}}=\left\vert H\right\vert
_{\Lip^{\gamma +2}}$ thanks to invariance of such Lip norms under
translation. The proof is then easily finished.
\end{proof}

\begin{appendixLemma}
  \label{lemConvergenceOfFlows}
  Assume the setting of the previous lemma. Assume that $\rp^n, n\ge 1$
  is a sequence of $p$ rough paths that converge to a rough path $\rp^0$ in $p$-variation.

  Then locally uniformly on $[0,T] \times \R^n \times \R$
  \begin{align*}
    (\phi^n, \frac{1}{\partial_y \phi^n}, \partial_y \phi^n, \partial_{yy} \phi^n, \partial_x \phi^n, \partial_{xx} \phi^n, \partial_{yx} \phi^n)
    \to
   (\phi^0, \frac{1}{\partial_y \phi^0}, \partial_y \phi^0, \partial_{yy} \phi^0, \partial_x \phi^0, \partial_{xx} \phi^0, \partial_{yx} \phi^0)
  \end{align*}
\end{appendixLemma}
\begin{proof}
  Using enlargment of the state space as in the proof of Lemma \ref{lemFlowEstimates} we can apply
  the same reasoning as in Theorem 11.14 and Theorem 11.15 in \cite{bibFrizVictoir} to get the desired result.
\end{proof}

\section{Comparison for PDEs I }
Consider the equation
\begin{align}
  \label{eqParabolicEquation}
  &\partial_t u(t,x) + F(t,x,u,Du,D^2u) = 0, \quad (t,x) \in [0,T) \times \R^n \\
  &u(T,x) = g(x), \quad x \in \R^n, \notag
\end{align}
where $F: [0,T] \times \R^n \times \R \times \R^n \times \mathcal{S}(n) \to \R$ is a continuous function
and $g: \R^n \to \R$ is a bounded, continuous function.

\begin{appendixTheorem}
  \label{thmComparisonForParabolic}
  Assume that $-F$ satisfies (3.14) of the User's Guide \cite{bibCrandallIshiiLions}, uniformly in $t$, together with uniform continuity of
  $F = F(t,x,r,p,X)$ whenever $r,p,X$ remain bounded.

  Assume also a (weak form of) properness: there exists $C$ such that
  \begin{align*}
    F(t,x,s,p,X) - F(t,x,r,p,X) \le C (s - r), \qquad \forall r \le s.
  \end{align*}

  If $u$ is a subsolution of \eqref{eqParabolicEquation} and $v$ is a supersolution, then for $(t,x) \in [0,T] \times \R^n$
  \begin{align*}
    u(t,x) \le v(t,x).
  \end{align*}
\end{appendixTheorem}
\begin{proof}
  Let $\tilde{u}(t,x) := u(T-t,x), \tilde{v}(t,x) := v(T-t,x)$.
  Then $\tilde{u}$ is a subsolution and $\tilde{v}$ is a supersolution to
  \begin{align*}
    \partial_t u(t,x) - F(t,x,u(t,x),Du(t,x),D^2u(t,x)) = 0,
    u(0,x) = g(x).
  \end{align*}

  Hence we can apply Theorem 20 in \cite{bibFrizOberhauser} to get the desired result (note that the $F$ there is $-F$ here, since we consider a terminal value problem).
\end{proof}

\section{ Comparison for PDEs II }
We consider the equation
\begin{align}
  \label{eq18}
  -\partial_t u - \frac{1}{2} \Tr[ \sigma(t,x) \sigma(t,x)^T D^2 u ] - \langle b(t,x), Du \rangle - f(t,x,u,Du \sigma(t,x) ) &= 0, \quad (t,x) \in [0,T] \times \R^n, \\
  u(T,x) &= g(x), \quad x \in \R^n, \notag
\end{align}
where $f: [0,T] \times \R^n \times \R \times \R^n \to \R$ is a continuous function
and $g: \R^n \to \R$ is a bounded, continuous function.

The following statement as well as its proof are a modification of Theorem 3.2 in \cite{bibKobylanski}.
(The statement is not in is most general form, but adjusted to what we need in the main text.)
\begin{appendixTheorem}
  \label{thmComparisonForPDE}
      Assume that there exists a constant $L>0$ such that for $(t,x) \in [0,T] \times \R^n$
      \begin{align*}
        |b(t,x) - b(t,y)| + |\sigma(t,x) - \sigma(y)| &\le L |x-y|, \\
        |b(t,x)|^2 + |\sigma(t,x)|^2 &\le L ( 1 + |x|^2 ).
      \end{align*}

      Assume that there exists a constant $C > 0$ such that for $(t,x,y,z) \in [0,T] \times \R^n \times \R \times \R^n$
      \begin{align*}
        |f(t,x,y,z)| &\le C ( 1 + |z|^2 ), \\ 
        |\partial_z f(t,x,y,z)| &\le C ( 1 + |z| ).
      \end{align*}
      Assume that there exists a constant $\Cunif$ such that for every $\varepsilon > 0$ there exists an
      $h_\varepsilon > 0$ such that for $(t,x,y,z) \in [T-h_\varepsilon,T] \times \R^n \times \R \times \R^n$ we have
      \begin{align}
        \label{eqPDEUniformBoundAssumption}
        \partial_y f(t,x,y,z) \le \Cunif + \varepsilon |z|^2.
      \end{align}
      Assume that there exists a constant $C > 0$ such that for $(t,x,y,z) \in [0,T] \times \R^n \times \R \times \R^n$
      \begin{align*}
        |\partial_x f(t,x,y,z)| \le C ( 1 + |z|^2 ).
      \end{align*}

  Then there exists an $\varepsilon^* = \varepsilon^*( ||u||_\infty, ||v||_\infty, C, \Cunif ) > 0$
  such
  that if
  $u$ is a bounded upper semicontinuous viscosity solution of \eqref{eq18} on $[T-h_{\varepsilon^*},T]$
  and if
  $v$ is a bounded lower semicontinuous viscosity solution of \eqref{eq18} on $[T-h_{\varepsilon^*},T]$
  such that for $x \in \R^n$
  \begin{align*}
    u(T,x) \le v(T,x),
  \end{align*}
  then for $(t,x) \in [T-h_{\varepsilon^*},T) \times \R^n$ we have
  \begin{align*}
    u(t,x) \le v(t,x).
  \end{align*}
\end{appendixTheorem}
\begin{proof}
  Set $M := \max\{ ||u||_{\infty}, ||v||_{\infty} \} + 1$.
  Let $\lambda > 0, A > 1, K > 0$ be constants to be chosen later.
  Define
  \begin{align*}
    \phi(\tu) := \frac{1}{\lambda} \ln\left( \frac{e^{\lambda A \tu} + 1}{A} \right): \R \to (-\frac{\ln( A )}{\lambda}, \infty).
  \end{align*}
  Since we want to plug $u$ and $v$ in the inverse of $\phi$ later on, we will have to choose
  $A \ge e^{ \lambda 2 M e^{Kt} }$,
  so that $\{ e^{Kt}(\uu - M) : \uu \in [-M,M] \}$ is contained in the range of $\phi$.
  Then
  \begin{align*}
    \phi'(\tu) = A \frac{1}{ 1 + e^{-\lambda A \tu} },
    \quad
    \phi^{-1}(\uu) = \frac{1}{\lambda A} \ln\left( A e^{\lambda \uu} - 1 \right).
  \end{align*}

  By differentiating $\phi( \phi^{-1}( \uu ) ) = \uu$ we get
  \begin{align*}
    (\phi^{-1})'( \uu ) = \frac{1}{ \phi'( \phi^{-1}( \uu ) ) },
    \quad
    (\phi^{-1})''(\uu) = - \frac{ \phi''( \phi^{-1}(\uu) ) }{ \phi'( \phi^{-1}(\uu)^2 }.
  \end{align*}
  Define
    $r(\uu) := \phi^{-1}( e^{Kt} (\uu - M) )$
  , its inverse
    $s(\tu) := \phi( \tu ) e^{-Kt} + M$
  and
    $g(\uu) := e^{ -\lambda e^{Kt} (\uu - M) }: [-M, M] \to [1, e^{\lambda 2 M e^{Kt}}]$.
  Then
    $g'(\uu) = -\lambda e^{Kt} g(\uu)$.

  Define
  \begin{align*}
    w(\uu)
    &:= e^{-K t} \phi'( r(\uu) )
    = \partial_{\tu} s |_{\tu = r(\uu)}
    = e^{-Kt} \left[ A - e^{ -\lambda e^{Kt} (\uu - M) } \right]
    = e^{-Kt} \left[ A - g(\uu) \right],
  \end{align*}
  which is non-negative for $A \ge e^{\lambda 2 M e^{Kt}}$.
  Then
  \begin{align*}
    w'(\uu)
    = \lambda g(\uu), \quad
    w''(\uu)
    = - e^{Kt} \lambda^2 g(\uu).
  \end{align*}

  Let now $u(t,x)$ be a solution to \eqref{eq18}. Let $\tilde{u}(t,x) := r(u(t,x))$.
  Then $u(t,x) = s( \tu(t,x) )$, and hence
  \begin{align*}
    \partial_{x_i} u(t,x)
    &= \phi'(\tilde{u}(t,x)) e^{-Kt} \partial_{x_i} \tilde{u}(t,x), \\
    \partial_{x_j x_i} u(t,x)
    &=
    \phi''(\tilde{u}(t,x)) e^{-Kt} \partial_{x_j} \tilde{u}(t,x) \partial_{x_i} \tilde{u}(t,x)
    + \phi'(\tilde{u}(t,x)) e^{-Kt} \partial_{x_j x_i} \tilde{u}(t,x),
  \end{align*}
  i.e.
  \begin{align*}
    D u(t,x) &= \phi'( \tilde{u}(t,x) ) e^{-K t} D \tilde{u}(t,x), \\
    D^2 u(t,x) &= \phi''(\tilde{u}(t,x)) e^{-K t} D\tilde{u}(t,x) \otimes D\tilde{u}(t,x) + \phi'(\tilde{u}(t,x)) e^{-K t} D^2 \tilde{u}(t,x).
  \end{align*}

  Hence
  \begin{align*}
    \partial_t \tilde{u}(t,x)
    &=
    \frac{1}{ \phi'( \tilde{u}(t,x) ) }
    \left[ K e^{Kt} (u(t,x) - M) + e^{Kt} \partial_t u(t,x) \right] \\
    &=
    \frac{1}{ \phi'( \tilde{u}(t,x) ) }
    K e^{Kt} (u(t,x) - M) \\
    &\quad
    - 
    \frac{1}{ \phi'( \tilde{u}(t,x) ) }
    e^{Kt} \left[ \frac{1}{2} \Tr[ \sigma(t,x) \sigma(t,x)^T D^2 u(t,x) ] + \langle b(t,x), Du(t,x) \rangle \right] \\
    &\quad
    -
    \frac{1}{ \phi'( \tilde{u}(t,x) ) }
    e^{Kt} f(t,x, u(t,x), D u(t,x) \sigma(t,x) ) \\
    &=
    K \frac{ \phi(\tilde{u}(t,x) }{ \phi'( \tilde{u}(t,x) ) } \\
    &\quad
    - \frac{1}{2} \Tr[ \sigma(t,x) \sigma(t,x)^T D^2 \tilde{u}(t,x) ]
    - \frac{\phi''(\tilde{u}(t,x))}{ \phi'( \tilde{u}(t,x) ) }
    \frac{1}{2} \Tr[ \sigma(t,x) \sigma(t,x)^T  D\tilde{u}(t,x) \otimes D\tilde{u}(t,x)] \\
    &\quad
    - \langle b(t,x),  D \tilde{u}(t,x) \rangle
    - \frac{1}{ \phi'( \tilde{u}(t,x) ) } e^{Kt} f(t, x, s( \tilde{u}(t,x) ), \phi'( \tilde{u}(t,x) ) e^{-Kt} D \tilde{u}(t,x) \sigma(t,x) )
  \end{align*}

  So $\tilde{u}$ is a solution to
  \begin{align*}
    -\partial_t \tilde{u}(t,x)
    -\frac{1}{2} \Tr[ \sigma(t,x) \sigma(t,x)^T D^2u(t,x) ]
    - \langle( b(t,x), D\tilde{u}(t,x) \rangle
    - \tilde{f}(t,x,\tilde{u}(t,x), D\tilde{u}(t,x) \sigma(t,x) )
    = 0,
  \end{align*}
  where, denoting from now on $\uu = s(\tu)$, $\pp = w(\uu) \tp$,
  \begin{align*}
    \tilde{f}(t,x,\tu,\tp)
    &=
    - K \frac{ \phi(\tu) }{ \phi'( \tu ) }
    + \frac{ \phi''(\tu) }{ \phi'( \tu ) }
      \frac{1}{2}
      |\tp|^2 \\
    &\quad
    + \frac{1}{ \phi'( \tu ) } e^{Kt} f(t, x, s( \tu ), \phi'( \tu ) e^{-Kt} \tp ) \\
    &=
    - K \frac{ \uu - M }{ w(\uu) }
    + w'(\uu) \frac{1}{2} |\tp|^2
    + \frac{1}{ w(\uu) } f(t, x, \uu, w(\uu) \tp ).
  \end{align*}

  Hence
  \begin{align*}
    \partial_{\tu} \tilde{f}(t,x,\tu, \tp )
    &=
    -
    K ( 1 - (\uu-M) \frac{w'(\uu)}{w(\uu)} ) )
    +
    \frac{1}{2} \frac{w''(\uu)}{w(\uu)} |\pp|^2  \\
    &\qquad
    -
    \frac{w'(\uu)}{w(\uu)} f(t, x, \uu, \pp )
    +
    \partial_\uu f(t, x, \uu, \pp \sigma(t,x) ) \\
    &\qquad
    +
    \frac{w'(\uu)}{w(\uu)} \partial_\pp f(t, x, \uu, \pp ) \pp \\
    &\le
    -
    K ( 1 - (\uu-M) \frac{w'(\uu)}{w(\uu)} ) )
    +
    \frac{1}{2} \frac{w''(\uu)}{w(\uu)} |\pp|^2 \\
    &\qquad
    +
    \frac{w'(\uu)}{w(\uu)} C ( 1 + |\pp|^2 )
    +
    \partial_\uu f(t, x, \uu, \pp ) \\
    &\qquad
    +
    \frac{w'(\uu)}{w(\uu)} C (1 + |\pp|) |\pp| \\
    &\le
    \frac{ |\pp|^2 }{w(\uu)}
    \left(
       \frac{1}{2} w''(\uu)
       +
       C w'(\uu)
       +
       C w'(\uu)
    \right) \\
    &\qquad
    - K ( 1 - (\uu-M) \frac{w'(\uu)}{w(\uu)} )
    + \partial_\uu f(t, x, \uu, \pp )
    + C \frac{w'(\uu)}{w(\uu)}
    + C \frac{w'(\uu)}{w(\uu)} |\pp|
  \end{align*}
  Now using
  \begin{align*}
    C \frac{w'(\uu)}{w(\uu)} |\pp|
    \le 
    \frac{ |\pp|^2 }{w(\uu)} w'(\uu)
    +
    \frac{w'(\uu)}{w(\uu)} C^2
  \end{align*}
  we get
  \begin{align}
    \label{eqPreFinalInequality}
    \partial_{\tu} \tilde{f}(t,x,\tu, \tp )
    &\le
    \frac{ |\pp|^2 }{w(\uu)}
    \left(
       \frac{1}{2} w''(\uu)
       +
       (2C + 1) w'(\uu)
    \right) \\
    &\qquad \qquad
    -
    K
    + \partial_\uu f(t, x, \uu, \pp )
    +
    \frac{w'(\uu)}{w(\uu)} \left( C + K (\uu-M) + C^2 \right). \notag
  \end{align}

  Note that
  \begin{align*}
    C + K_0 (\uu-M) + C^2
    &\le C - K_0 + C^2, \quad \uu \in [-(M-1),M-1].
  \end{align*}
  Hence we can choose $K_0 = K_0(C)$ sufficiently large, such that
  \begin{align*}
    C + K_0 (\uu-M) + C^2 &\le -1, \quad \uu \in [-(M-1),M-1].
  \end{align*}
  Then we have that for all choices of $K_0 > K$, and all choices $\lambda > 0$
  that the last term in \eqref{eqPreFinalInequality}
  \begin{align*}
    \frac{w'(\uu)}{w(\uu)} \left( C + K (\uu-M) + C^2 \right)
    =
    \frac{ e^{-\lambda e^{Kt} (\uu-M)} }{ A - e^{-\lambda e^{Kt} (\uu-M)} } \lambda e^{Kt}
    \left( C + K (\uu-M) + C^2 \right)
  \end{align*}
  is negative, as long as $A > e^{\lambda 2 M e^{Kt}}$.
  We now fix $K = K(C, \Cunif) = \max\{ K_0(C), \Cunif \} + 1$.
  Then
  \begin{align*}
      \frac{1}{2} w''(\uu) + (2C + 1) w'(\uu)
    &=
      -\frac{1}{2} e^{Kt} \lambda^2 g(\uu) + \lambda (2 C + 1) g(\uu) \\
    &\le
    -\frac{1}{2} \lambda^2 g(\uu) + \lambda (2 C + 1) g(\uu) \\
    &=
    g(\uu) \lambda \left[ (2 C + 1) - \frac{1}{2} \lambda \right].
  \end{align*}

  So, if we choose $\lambda = \lambda(C) = 4C + 4$, we have
  \begin{align*}
    \frac{1}{2} w''(\uu) + (2C + 1) w'(\uu)
    &=
    g(\uu) (4C + 4) (-1)
    \le
    -(4C + 4)
    \le
    -1.
  \end{align*}

  We now fix $A = A(\lambda(C), M, K(C, \Cunif) = A(M, C, \Cunif) = e^{\lambda 2 M e^{K T}} + 1$.
  Then for the first term in \eqref{eqPreFinalInequality}
  \begin{align*}
    \frac{ |\pp|^2 }{w(\uu)}
    \left(
       \frac{1}{2} w''(\uu)
       +
       (2C + 1) w'(\uu)
    \right)
    &=
    \frac{ |\pp|^2 }{ e^{\lambda 2 M e^{K T}} + 1 - e^{-\lambda e^{Kt} (\uu-M)} } e^{Kt}
    \left(
       \frac{1}{2} w''(\uu)
       +
       (2C + 1) w'(\uu)
    \right) \\
    &\le
    - \frac{ |\pp|^2 }{ e^{\lambda 2 M e^{K T}} + 1 - e^{-\lambda e^{Kt} (\uu-M)} } e^{Kt} \\
    &\le
    - \frac{ |\pp|^2 }{ e^{\lambda 2 M e^{K T}} } e^{Kt} \\
    &< - \delta |\pp|^2 < 0,
  \end{align*}
  with
  \begin{align*}
    \delta
    = \delta(\lambda(C), K(C, \Cunif), M)
    = \delta(M, C, \Cunif)
    = \frac{ e^{Kt} }{ e^{\lambda 2 M e^{K T}} + 1 } > 0.
  \end{align*}

  If we now choose in \eqref{eqPDEUniformBoundAssumption} the $h = h(\delta) = h(M, C, \Cunif) > 0$ so small, so that on $[T-h,T]$ we have
  \begin{align*}
    \partial_y f(t,x,y,z) \le \Cunif + \frac{\delta}{2} |z|^2,
  \end{align*}
  we get that on $[T-h,T]$
  \begin{align*}
    \partial_{\tu} \tilde{f}(t,x,\tu, \tp )
    &\le
    \frac{ |\pp|^2 }{w(\uu)}
    \left(
       \frac{1}{2} w''(\uu)
       +
       (2C + 1) w'(\uu)
    \right) \\
    &\quad
    -
    K
    +
    \partial_\uu f(t, x, \uu, \pp )
    +
    \frac{w'(\uu)}{w(\uu)} \left( C + K (\uu-M) + C^2 \right) \\
    &\le
    - \delta |\pp|^2
    -
    K
    +
    \Cunif
    + \frac{\delta}{2} |\pp|^2 \\
    &\le
    - \delta |\pp|^2
    + \frac{\delta}{2} |\pp|^2
    - 1 \\
    &=
    - \frac{\delta}{2} |\pp|^2
    - 1.
  \end{align*}
  Which is the desired inequality, (23) in \cite{bibKobylanski}.

  The rest of the proof is now an exact copy of the proof of Theorem 3.2 in \cite{bibKobylanski}.
\end{proof}

{\bf Acknowledgement:} Part of this work was carried out during the 2010 SPDE programme at INI, Cambridge. The first author is supported by an IRTG (Berlin-Zurich) PhD-scholarship.

\bibliographystyle{plain}
\bibliography{bsdes-with-rough-drivers}

\end{document}